\documentclass[12pt]{amsart}
\usepackage[dvips]{epsfig}
\usepackage{graphics}
\usepackage{latexsym}
\usepackage{verbatim}
\usepackage{amsmath}
\usepackage{amsthm}
\usepackage{amssymb}
\usepackage{hyperref}
\usepackage[]{hyperref}
\hypersetup{
    colorlinks=true,%
    filecolor=black,%
    linkcolor=black,%
    urlcolor=black  %
}
\usepackage [table]{xcolor}
\usepackage{multirow}
\usepackage{float}
\usepackage{tikz}
\usepackage{subfig}

\graphicspath{{images/}}	
\newtheorem{theorem}{Theorem}[section]

\newtheorem{proposition}[theorem]{Proposition}

\newtheorem{remark}[theorem]{Remark}

\newcommand\CC{\hbox{C\kern -.58em {\raise .54ex \hbox{$\scriptscriptstyle |$}}
  \kern-.55em {\raise .53ex \hbox{$\scriptscriptstyle |$}} }}
\newcommand\NN{\hbox{I\kern-.2em\hbox{N}}}
\newcommand\RR{\mathbb{R}}

\newcommand\ZZ{{{\rm Z}\kern-.28em{\rm Z}}}
\newcommand\Gradx{ \nabla_{\mathbf{x}}}
\newcommand\Gradxp{ \nabla_{\mathbf{x}_\perp}}
\newcommand\Gradv{ \nabla_{\mathbf{v}}}
\newcommand\Gradvp{ \nabla_{\mathbf{v}_\perp}}
\newcommand\Div{ \textrm{div}}

\newcommand\jj{\mathbf{J}}
\newcommand\xx{ \mathbf{x} }

\newcommand\zz{ \mathbf{z} }
\newcommand\XX{ \mathbf{X} }
\newcommand\VV{ \mathbf{V} }

\newcommand\s{\mathbf{s}}
\newcommand\yy{ \mathbf{y} }
\newcommand\kk{ \mathbf{k} }
\newcommand\vv{ \mathbf{v} }
\newcommand\ds{ \displaystyle }
\newcommand\vp{ \varphi }

\def\eps{\varepsilon }
\renewcommand\d{\partial}
\newcommand{\Id}{{\rm Id}}

\newcommand\bB{{\mathbf B}}

\newcommand\bE{{\mathbf E}}
\newcommand\bF{{\mathbf F}}

\newcommand\bW{{\mathbf W}}
\newcommand\bX{{\mathbf X}}
\newcommand\bY{{\mathbf Y}}
\newcommand\bZ{{\mathbf Z}}

\def\signFF{\bigskip\bigskip\hspace{80mm}
\vbox{{\sc Francis Filbet\par\vspace{3mm}
Universit\'e de Toulouse III \& IUF \par
UMR5219, Institut de Math\'ematiques de Toulouse,\par
118, route de Narbonne\par
F-31062 Toulouse cedex,  FRANCE
\par\vspace{3mm}e-mail:} francis.filbet@math.univ-toulouse.fr }}

\def\signLMR{\bigskip\bigskip\hspace{80mm}
\vbox{{\sc  Luis Miguel Rodrigues \par\vspace{3mm}
Universit\'e de Rennes 1,\par
UMR6625, IRMAR,\par
263 avenue du General Leclerc,\par
F-35042 Rennes Cedex,  FRANCE
\par\vspace{3mm}e-mail:}  luis-miguel.rodrigues@univ-rennes1.fr }}

\begin{document}

\title[Particle-In-Cell Methods for the
  Vlasov-Poisson System]{Asymptotically stable particle-in-cell methods for the
  Vlasov-Poisson system with a strong external magnetic field}

\author{Francis Filbet, Luis Miguel Rodrigues}

\maketitle

\begin{abstract}
This paper deals with the numerical resolution of the Vlasov-Poisson
system with a strong external magnetic field  by Particle-In-Cell
(PIC) methods. In this regime, classical PIC methods are subject to
stability constraints on the time and space steps related to the small
Larmor radius and plasma frequency. Here, we propose an
asymptotic-preserving PIC scheme which is not subjected to these
limitations.  Our approach is based on first and higher order semi-implicit numerical
schemes already validated on dissipative systems \cite{BFR:15}. Additionally, when the magnitude of the external magnetic
field becomes large, this method provides a consistent PIC
discretization of the guiding-center equation, that is, incompressible
Euler equation in vorticity form. We propose several numerical experiments which provide a solid validation of the method and its underlying concepts.

\end{abstract}

\vspace{0.1cm}


\tableofcontents

\section{Introduction} 
\setcounter{equation}{0}
\label{sec:1}
Magnetized plasmas are encountered in a wide variety of astrophysical
situations, but also  in magnetic fusion devices such as tokamaks,
where a large external magnetic field needs to be applied in order to keep
 the particles on the desired tracks. 

In Particle-In-Cell (PIC) simulations of such devices, this large
external magnetic field obviously needs to be taken into account when
pushing the particles. However, due to the magnitude of the concerned
field this often adds a new time scale to the simulation and thus a
stringent restriction on the time step. In order to get rid of this
additional time scale, we would like to find approximate equations,
where only the gross behavior implied by the external field would be
retained and which could be used in a numerical simulation. In the
simplest situation, the trajectory of a particle in a constant magnetic field $\mathbf{B}$ is a
helicoid along the magnetic field lines with a radius proportional to the inverse of the magnitude of  $\mathbf{B}$. Hence, when this field becomes very
large the particle gets trapped along the magnetic field
lines. However due to the fast oscillations around the apparent
trajectory, its apparent velocity is smaller than the actual one. This
result has been known for some time as the guiding center approximation, and the link between the real and the apparent velocity
is well known in terms of $\mathbf{B}$.

Here, we consider a plasma constituted of a large number of charged
particles, which is described by the Vlasov equation coupled with the Maxwell or Poisson
equations to compute the self-consistent fields. It describes the
evolution of a system of particles under the effects of external and
self-consistent fields. The unknown $f(t,\xx,\vv)$, depending on the
time $t$, the position $\xx$, and the velocity $\vv$, represents the
distribution of particles in phase space for each species with $(\xx,\vv) \in \RR^d\times \RR^d$, $d=1,..,3$. Its behaviour is given by
the Vlasov equation, 
\begin{equation}
\label{eq:vlasov} 
\frac{\partial f}{\partial t}\,+\,\mathbf{v}\cdot\Gradx f \,+\,\mathbf{F}(t,\mathbf{x},\mathbf{v})\cdot\Gradv f \,=\, 0,
\end{equation}
where the force field $F(t,\xx,\vv)$ is coupled with the distribution function $f$ giving a nonlinear system.  
We first define  $\rho(t,\xx)$  the charge density  and  $\jj (t,\xx)$ the current density which are given by
\begin{equation*}
\rho(t,\xx) = q\int_{\RR^d} f(t,\xx,\vv)d\vv,\quad \jj(t,\xx) = q\int_{\RR^d} \vv\,f(t,\xx,\vv)d\vv,
\end{equation*}
where $q$ is the elementary charge. For the Vlasov-Poisson model 
\begin{equation}
\label{poisson}
\mathbf{F}(t,\xx,\vv) = \frac{q}{m}\,\mathbf{E}(t,\xx),\quad \mathbf{E}(t,\xx) = -\nabla_{\xx} \phi(t,\xx), \quad - \Delta_\xx\phi = \frac{\rho}{\varepsilon_0},
\end{equation}
where $m$ represents the mass of one particle. On the other hand for the Vlasov-Maxwell model, we have
\begin{equation*}
\mathbf{F}(t,\xx,\vv) = \frac{q}{m}\,(\mathbf{E}(t,\xx) + \vv \wedge \mathbf{B}(t,\xx)\,),
\end{equation*}
and $\mathbf{E}$, $\mathbf{B}$ are solutions of the Maxwell equations
\begin{equation*}
\left\{
\begin{array}{l}
\displaystyle{\frac{\partial \mathbf{E}}{\partial t} \,-\, c^2  \, \nabla \times \bB  \,=\, -\frac{\jj}{\varepsilon_0},}
\\
\,
\\
\displaystyle{\frac{\partial \mathbf{B}}{\partial t} \,+\, \nabla \times \mathbf{E} \,=\, 0,}
\\
\,
\\
\displaystyle{\nabla \cdot \mathbf{E} = \frac{\rho}{\varepsilon_0}, \quad \nabla \cdot \mathbf{B} = 0,}
\end{array}\right.
\end{equation*}
with the compatibility condition
\begin{equation*}
\frac{\partial \rho}{\partial t} + \Div_\xx \jj = 0,
\end{equation*}
which is verified by 
the solutions of the Vlasov equation.

Here we will consider an intermediate model where the magnetic field
is given,  
$$
{\bf B}(t,\xx) \,\,=\,\, \frac{1}{\eps} \,{\bf B}_{\rm ext}(t,\xx_\perp), 
$$
with  $\varepsilon >0$ and we focus on the long time behavior of the plasma in the orthogonal plane to the external
magnetic field,  that is the two dimensional Vlasov-Poisson
system with an external strong magnetic field
\begin{equation}
\label{eq:vlasov2d}
\left\{
\begin{array}{l} 
\displaystyle{\varepsilon\frac{\partial f}{\partial t}\,+\,\mathbf{v}_\perp\cdot\Gradxp f \,+\,\left(
  \mathbf{E}(t,\xx_\perp) \,+\, \frac{1}{\eps}\vv_\perp \wedge {\bf B}_{\rm ext}(t,\xx_\perp) \right)\cdot\Gradvp f
\,=\, 0, \quad (\xx_\perp,\vv_\perp)\in\RR^4},
\\
\,
\\
\displaystyle{E = - \Gradxp \phi, \quad -\Delta_{\xx_\perp} \phi = \rho,  \quad \xx_\perp\in\RR^2}.
\end{array}\right.
\end{equation}
Here, for simplicity we 
set all physical constants 
to one and consider that $\varepsilon>0$ is
a small parameter related to the ratio between the reciprocal Larmor
frequency and the advection time scale. The term $\varepsilon$ in
front of the time derivative of $f$ 
stands for the fact 
that we want to approximate
the solution for large time.

We want to  construct numerical solutions to the Vlasov-Poisson system (\ref{eq:vlasov2d}) by particle methods (see \cite{birdsall}), which consist in
approximating the distribution function by a finite number of
macro-particles. The trajectories of these particles are computed from
the characteristic curves corresponding to the Vlasov equation
\begin{equation}
\label{traj:00}
\left\{
\begin{array}{l}
\ds{\varepsilon\frac{d\XX}{dt} = \VV,} 
\\
\,
\\
\ds{\varepsilon\frac{d\VV}{dt} = \frac{1}{\varepsilon}\VV\wedge {\bf
    B}_{\rm ext}(t,\XX)    \,+\, {\bf E}(t,\XX), }
\\
\,
\\
\XX(t^0) = \xx^0_\perp, \, \VV(t^0) = \vv^0_\perp,
\end{array}\right.
\end{equation}
where the electric field is computed from a discretization of the Poisson equation
in (\ref{eq:vlasov2d}) on a mesh of the physical space.
  
The main purpose of this work is the construction of efficient
numerical methods for stiff transport equations of type (\ref{eq:vlasov2d}) in the
limit $\varepsilon\rightarrow 0$.  Indeed,  setting 
$$
\bZ \,=\, \frac{\VV}{\eps} - \frac{\bE\wedge\bB_{\rm
    ext}}{\|\bB_{\rm ext}\|^2}, 
$$
the  system (\ref{traj:00}) can be re-written for $(\XX, \bZ)$ as
\begin{equation}
\label{traj:01}
\left\{
\begin{array}{l}
\ds{\frac{d\XX}{dt} \,=\,  \frac{\bE \wedge\bB_{\rm
      ext}}{\| \bB_{\rm ext}\|^2}  \,+\,\bZ,} 
\\
\,
\\
\ds{\frac{d\bZ}{dt} \,=\, -\frac{1}{\eps^2}\, {\bf B}_{\rm ext} \wedge
  \bZ \,-\, \frac{d}{dt}\left(\frac{\bE\wedge\bB_{\rm
    ext}}{\|\bB_{\rm ext}\|^2}\right), }
\\
\,
\\
\XX(t^0) = \xx^0_\perp, \quad \bZ(t^0) = \ds\frac{1}{\eps}\vv^0_\perp - \frac{\bE\wedge\bB_{\rm
    ext}}{\|\bB_{\rm ext}\|^2}(t^0,\xx_\perp^0).
\end{array}\right.
\end{equation}
Therefore, we denote by $(\bX^\eps,\bZ^\eps)$ the solution to
(\ref{traj:01}), and under some classical smoothness assumptions on the
electromagnetic fields $(\bE, \bB_{\rm ext})$,  
it is well-known at least when $\vv^0_\perp=0$ or $\bB_{\rm ext}$ is homogeneous 
that  $(\bZ^\eps)_{\eps>0}$ converges 
weakly 
to zero when $\eps\rightarrow 0$, and 
$\XX^\eps \rightharpoonup \bY$, 
where $\bY$ corresponds to the guiding center
approximation
\begin{equation}
\label{traj:limit}
\left\{
\begin{array}{l}
\ds{\frac{d\bY}{dt} \,=\,  \frac{\bE \wedge\bB_{\rm
      ext}}{\| \bB_{\rm ext}\|^2}(t,\bY),} 
\\
\,
\\
\bY(t^0) = \xx^0_\perp.
\end{array}\right.
\end{equation}
Here, we are of course interested in the behavior of the sequence solution $(f^\eps)_{\eps>0}$ to the
Vlasov-Poisson system (\ref{eq:vlasov2d}) when
$\varepsilon\rightarrow 0$, which corresponds to the gyro-kinetic
approximation of the Vlasov-Poissons system.  Following the work of L.
Saint-Raymond \cite{SR:00},  it can be proved 
--- at least when $\bB_{\rm ext}$ is homogeneous --- 
that the charge density
$(\rho^\eps)_{\eps>0}$ converges to the  solution to the guiding center approximation
\begin{equation}
\left\{
\begin{array}{l}
 \displaystyle{\frac{\partial \rho}{\partial t}+\mathbf{U}\cdot\nabla\rho=0},\\[3.5mm]
 -\Delta_{\xx_\perp}\phi=\rho,
\end{array}
\right.
   \label{eq:gc}
  \end{equation}
where the velocity $\mathbf{U}$ is
$$
{\bf U} = \frac{\bE \wedge \bB_{\rm ext}}{\|\bB_{\rm ext} \|^2}, \quad
\bE = -\Gradxp \phi.
$$  
We observe that the limit system (\ref{traj:limit}) corresponds to the
characteristic curves to the limit equation (\ref{eq:gc}).

We seek a method that is able to capture these properties, while the numerical
parameters may be kept independent of the stiffness degree of these
scales. This concept is known and widely studied for dissipative
systems in the framework of asymptotic preserving schemes
\cite{jin:99, klar:98}.  Contrary to collisional kinetic equations in hydrodynamic or
diffusion asymptotic, collisionless equations like the Vlasov-Poisson
system (\ref{eq:vlasov2d}) involve time oscillations. In this context,
the situation is more complicated than the one encountered in
collisional regimes since we
cannot expect any dissipative phenomenon. Therefore, the notion of
two-scale convergence has been introduced both at the theoretical and
numerical level \cite{CFHM:15, FSS:09, FS:00}  in order to derive asymptotic
models. However, these asymptotic models, obtained after removing the
fast scales,  are valid only when $\varepsilon$ is
small.  We refer to E. Fr\'enod and E. Sonnendr\"ucker \cite{FS:00}, and
F. Golse and L. Saint-Raymond \cite{GSR:99,SR:02} for a theoretical point of
view on these questions, and E. Fr\'enod, F. Salvarani and E. Sonnendr\"ucker \cite{FSS:09} for numerical applications of such techniques.

Another approach is to combine both disparate scales into one and
single model.  Such a decomposition can be done using  a micro-macro
approach (see \cite{CFHM:15} and the references therein). Such a model
may be used when the small parameter of the equation is not everywhere small. Hence, a scheme for a micro-macro model can switch
from one regime to another without any treatment of the transition
between the different regimes.  A different method consists in separating fast and slow time scales
when such a structure can be identified \cite{bibCLM} or
\cite{FHLS:15}. 

Theses techniques work well when the magnetic field is uniform since
fast scales can be computed from a formal
asymptotic analysis, but for more complicated problems, that is, when
the external magnetic field depends on time and position $\xx$, the
generalization of this approach is an open problem.  

In this paper, we propose an alternative to such methods allowing to
make direct simulations of systems (\ref{eq:vlasov2d}) with large time steps with respect to $\varepsilon$. We develop numerical schemes that are able to
deal with a wide range of values for $\varepsilon$,  so-called Asymptotic Preserving (AP) class
\cite{klar:98, jin:99}, such schemes are consistent with the kinetic model for all positive value
of $\varepsilon$, and degenerate into consistent schemes with the asymptotic model
when $\varepsilon\rightarrow0$.

Before 
presenting 
our time discretization technique, let us first briefly review the basic tools  of particle-in-cell methods which are widely used for plasma physics simulations \cite{birdsall}.

\section{A brief review of particle methods}
\label{sec:2}

The numerical resolution of the Vlasov equation and related models is usually performed by
Particle-In-Cell (PIC) methods which approximate the plasma by a
finite number of particles. Trajectories of these particles are
computed from characteristic curves (\ref{traj:00}) corresponding to
the the Vlasov equation (\ref{eq:vlasov2d}), whereas self-consistent fields are computed on a mesh of the physical
space.

This method yields satisfying results with a relatively small
number of particles but it is sometimes subject to fluctuations, due to the
numerical noise, which are difficult to control. To improve the
accuracy, direct numerical simulation
techniques have been developed.  The Vlasov equation is discretized
in phase space  using either semi-Lagrangian \cite{FSB, bibFS, bibQS11, bibSR}, finite difference
\cite{filbet-yang} or discontinuous Galerkin \cite{bianca,heath12}
schemes. But these direct methods are very costly, hence several
variants of particle methods  have been developed over the past
decades. In the Complex Particle Kinetic  scheme introduced by Bateson and Hewett \cite{ref2,
  ref18}, particles have a Gaussian shape that is transformed by the
local shearing of the flow. Moreover they can be fragmented to probe
for emerging features, and merged where fine particles are no longer
needed. In the Cloud in Mesh (CM) scheme of Alard and Colombi \cite{ref1} particles also have Gaussian shapes, and they are
deformed by local linearization of the force field. More recently in
\cite{ref7}, the authors proposed a Linearly-Transformed
Particle-In-Cell method, that employs linear deformations of the
particles.  
 
Here we focus on the time discretization technique, hence  we will only consider standard particle method even if
our approach is completely independant on the choice of the particle
method.

The particles method consists in approximating the initial condition $f_0$ in (\ref{eq:vlasov2d}) by the following Dirac mass sum
$$
f_N^0(\xx,\vv) \,:=\; \sum_{1\leq k\leq N } \omega_k \;\delta(\xx-\xx^0_k) \;\delta(\vv-\vv^0_k)\,,
$$
where $(\xx^0_{k},\vv_{k}^0)_{1 \leq k \leq N}$ is a beam of $N$ particles distributed in the four dimensional phase space according to the density function $f_0$. Afterwards, one
approximates the solution of (\ref{eq:vlasov2d}), by
$$
f_N(t,\xx,\vv) \,:=\; \sum_{1\leq k\leq N } \omega_k \;\delta(\xx-\XX_k(t)) \;\delta(\vv-\VV_k(t))\,,
$$
where $(\XX_k,\VV_k)_{1\leq k\leq N}$ is the position in phase space
of particle $k$ moving along the characteristic curves (\ref{traj:00})
with the initial data $(\xx_k^0, \vv_k^0)$, for $1\leq k \leq N$.
   
However when the Vlasov equation is coupled with the Poisson equation
for the computation of the electric field, the Dirac mass has to be
replaced by a smooth function $\varphi_\alpha$
$$
f_{N,\alpha}^0(\xx,\vv) \,:=\; \sum_{1\leq k \leq N} \omega_k
\;\vp_\alpha (\xx-\xx^0_k) \;\vp_\alpha (\vv-\vv^0_k)\,, 
$$
where $\vp_\varepsilon = \varepsilon^{-d} \vp(\cdot / \varepsilon)$ is a particle shape function with radius proportional to $\varepsilon$, 
usually seen as a smooth approximation of the Dirac measure obtained
by scaling a compactly supported ``cut-off'' function $\vp$ for which common choices include B-splines and smoothing kernels with vanishing moments, 
see {\it e.g.} \cite{Koumoutsakos.1997.jcp,Cottet.Koumoutsakos.2000.cup}.

Particle centers are then pushed forward at each time $t$ by following
a numerical approximation of the flow (\ref{traj:00}),  leading to
$$
f_{N,\alpha}(t,\xx,\vv) \,:=\, \sum_{1\leq k \leq N} \omega_k\, \vp_\alpha\left(\xx-\XX_k(t)\right)\, \vp_\alpha\left(\vv-\VV_k(t)\right).
$$
In the classical error analysis \cite{Beale.Majda.1982b.mcomp,Raviart.1985.lnm}, the above process is
seen as 
\begin{itemize}
\item An approximation (in the distribution sense) of the initial data by a collection of weighted Dirac measures\;;
\item The exact transport of the Dirac particles along the flow\,; 
\item The smoothing of the resulting distribution 
$$
\sum_{1\leq k\leq N} \omega_k \,\delta(.-\XX_k(t))
\,\delta(.-\VV_k(t))
$$ 
with the convolution kernel $\vp_\varepsilon$.
\end{itemize}

The classical error estimate reads then as follows \cite{Cohen.Perthame.2000.simath}: 
\begin{proposition}
\label{lmm:0}
Consider the Vlasov equation with a given electromagnetic field
$(\bE,\bB_{\rm ext})$ and a smooth initial datum $f^0\in C_c^{s}(\RR^d)$,
with $s\geq1$.  

If for some prescribed integers $m>0$ and $r>0$, the cut-off $\vp\geq 0$ has
$m$-th order smoothness and satisfies 
a
moment condition of order
$r$, namely, 
$$
\int_{\RR^d}  \vp(\yy) \,d\yy \,=\, 1, \quad\int_{\RR^d}  |\yy|^r \,\vp(y) \,d \yy \,<\, \infty,
$$ 
and
$$
\int_{\RR^d} y_1^{s_1} \ldots y_d^{s_d} \,\vp(\yy) \,d\yy \,=\, 0, 
\quad \text{ for } ~ \s=(s_1,\ldots,s_d) \in \NN^d ~ \text{ with } ~  1 \le s_1 + \cdots + s_d \le r-1.
$$
Then there exists a constant $C$ independent of $f_0$, $N$ or
$\alpha$, such that  we have  for all $1\leq p \leq +\infty$, 
$$
\| f(t)-f_{N,\alpha}(t) \|_{L^p} \rightarrow 0,
$$ 
when $N\rightarrow \infty$ and $\alpha\rightarrow 0$  where the ratio $N^{1/d} \alpha\ll 1$. 
\end{proposition}

Note that following \cite{Cohen.Perthame.2000.simath}, it is also possible to get explicit order of
convergence for the linear Vlasov equation.  Let us also mention related 
papers where the convergence of a numerical scheme for the Vlasov-Poisson 
system is investigated. G.-H. Cottet and P.-A. Raviart \cite{cottet_1} present a precise
mathematical analysis of the particle method for solving the one-dimensional 
Vlasov--Poisson system. We also mention the papers of
S. Wollman and E. Ozizmir \cite{wollman1_1} and S. Wollman \cite{wollman2_1} on the topic. 
K. Ganguly and H.D. Victory give a convergence result for the Vlasov-Maxwell
system \cite{ganguly_1}.

The rest of the paper is organized as follows. In Section \ref{sec:3}
we present several time discretization techniques based on high-order
semi-implicit schemes \cite{BFR:15} for  the Vlasov-Poisson system with a
strong external magnetic field, and we prove uniform consistency of
the schemes in the limit $\eps\rightarrow 0$ with preservation of the
order of accuracy (from first to third order accuracy). In Section \ref{sec:4} we perform a rigorous analysis of the
first order scheme for smooth electromagnetic fields.
 
Section \ref{sec:5} is then devoted to numerical simulations for one
single particle motion and for the Vlasov-Poisson model for various
asymptotics $\eps \approx 1$ and $\eps\ll 1$, which illustrate the
advantage of high order schemes.

\section{A particle method for Vlasov-Poisson
  system with
  a strong magnetic
  field}
\setcounter{equation}{0}
\label{sec:3}
Let us now consider the system (\ref{eq:vlasov2d}) and apply a
particle method, where the key issue is to design a uniformly stable
scheme with respect to the parameter $\eps>0$, which is related to the magnitude
of the external magnetic field. Assume that at time $t^n=n\,\Delta t$,
the set of particles are located in $(\xx_k^n, \vv_k^n)_{1\leq k\leq
  N}$, we want to solve the following system on the time interval $[t^n, t^{n+1}]$,
\begin{equation}
\label{traj:bis}
\left\{
\begin{array}{l}
\ds\eps\frac{d\XX_k}{dt} = \VV_k, 
\\
\,
\\
\ds\eps\frac{d\VV_k}{dt} = \frac{1}{\eps} \VV_k \wedge {\bf B}_{\rm ext}(t,\XX_k)  + {\bf E}(t,\XX_k), 
\\
\,
\\
\XX(t^n) = \xx^n_k, \, \VV(t^n) = \vv^n_k,
\end{array}\right.
\end{equation}
where the electric field is computed from a discretization of the Poisson equation
(\ref{eq:vlasov2d}) on a mesh of the physical space.
  
The numerical scheme that we describe here is proposed in the
framework of Particle-In-Cell method, where the solution $f$  is
discretized as follows
$$
f_{N,\alpha}^{n+1}(\xx,\vv) := \sum_{1\leq k \leq N} \omega_k\, \vp_\alpha
(\xx-\xx^{n+1}_k)\, \vp_\alpha(\vv-\vv^{n+1}_k), 
$$
where  $(\xx_k^{n+1},\vv_k^{n+1})$ represents an approximation of
the solution $\XX_k(t^{n+1})$ and $\VV_k(t^{n+1})$ to (\ref{traj:bis}).

When the Vlasov equation (\ref{eq:vlasov2d}) is coupled with the
Poisson equation (\ref{poisson}), the electric field is computed in a
macro-particle position $\xx^{n+1}_\kk$ at time $t^{n+1}$ as follows
  
\begin{itemize}
\item Compute the density $\rho$
$$
\rho_{h,\varepsilon}^{n}(\xx) = \sum_{\kk\in\ZZ^d} w_\kk \,\;\vp_\varepsilon (\xx-\xx^{n}_\kk).
$$
\item Solve a discrete approximation to (\ref{poisson}) 
$$
-\Delta_{h} \phi^n (\xx) = \rho_{h,\varepsilon}^{n}(\xx).
$$
\item Interpolate the electric field  with the same order of accuracy
  on the points $(\xx_\kk^{n})_{\kk\in\ZZ^d}$.
\end{itemize}

To discretize the system (\ref{traj:bis}),  we apply the strategy developed in
\cite{BFR:15} based on semi-implicit solver for stiff problems. In the
rest of this section, we propose several numerical 
schemes 
to the system (\ref{traj:bis}) for which the index $k\in\{1,\ldots,N\}$ will
be omitted.   

\subsection{A first order semi-implicit scheme}
We start with the simplest semi-implicit scheme for (\ref{traj:bis}), which is a
combination of backward and forward Euler scheme.  It gives for a fixed time step $\Delta t>0$ and a given
electric field ${\bf E}$ and an external magnetic field $\bB_{\rm ext}$,
\begin{equation}
\label{scheme:0}
\left\{
\begin{array}{l}
\ds\frac{\xx^{n+1} - \xx^n }{\Delta t} \,\,=\, \frac{\vv^{n+1}}{\eps},
\\
\,
\\
\ds\frac{\vv^{n+1} - \vv^n }{\Delta t} \,\,=\, \frac{1}{\eps}\left(
\frac{\vv^{n+1}}{\eps} \wedge \bB_{\rm ext} (t^n, \xx^n) + {\bE(t^n,\xx^n)}\right).
\end{array}\right.
\end{equation}
Note that only the second equation on $\vv^{n+1}$ is fully implicit
and 
requires 
the inversion of a linear operator. Then, from
$\vv^{n+1}$ the first equation gives the value for the position $\xx^{n+1}$.

\begin{proposition}[Consistency in the
  limit $\eps\rightarrow 0$ for a fixed $\Delta t$]
\label{prop:1}
Let us consider a time step $\Delta t>0$, a final time $T>0$  and
set $N_T=[T/\Delta t]$. Assume that the sequence $(\xx^n_\eps,\vv^n_\eps)_{0\leq
  n\leq N_T}$ given by
(\ref{scheme:0}) is such that for all $1\leq n\leq N_T$,
$\left(\xx^n_\eps,\eps\vv^n_\eps\right)_{\eps>0}$ is uniformly
bounded with respect to $\eps>0$ and $\left(\xx^0_\eps,\eps\vv^0_\eps\right)_{\eps>0}$ converges in the limit $\eps\rightarrow 0$ to some $(\yy^0,0)$. 
Then, for $1\leq n\leq N_T$, $\xx^n_\eps\rightarrow \yy^n$, as $\eps\rightarrow 0$ and the
limit $(\yy^n)_{1\leq n\leq N_T}$ is a consistent first order approximation with respect to $\Delta t$ of the
guiding center equation provided by the scheme
\begin{equation}
\label{sch:y0}
\frac{\yy^{n+1} - \yy^n}{\Delta t} = \bE(t^n,\yy^n)\wedge \frac{\bB_{\rm ext}(t^n,
\yy^n)}{\|\bB\|^2}.
\end{equation}
\end{proposition}
\begin{proof}
For all $1\leq n\leq N_T$, we consider $(\xx^n_\eps,\vv^n_\eps)$ the
solution to (\ref{scheme:0}) now labeled with respect to $\eps>0$.
Since, the sequence $(\xx^n_\eps)_{\eps>0}$ is uniformly bounded with respect to
$\eps>0$, we can extract a subsequence still abusively labeled by $\eps$ and find some 
$(\yy^n)$ such that $\xx^n_\eps\rightarrow \yy^n$ as $\eps$ goes to zero.
Then, we observe that the second equation of (\ref{scheme:0}) can be written as
$$
\eps^2\frac{\vv^{n+1}_\eps - \vv^n_\eps }{\eps\Delta t} \,\,=\, \left(\frac{\vv^{n+1}_\eps}{\eps} \wedge \bB_{\rm ext}(t^n, \xx^n_\eps) + \bE(t^n,\xx^n_\eps)\right).
$$
and that, for any $0\leq n\leq N_T$, $\left(\eps\vv^n_\eps\right)_{\eps>0}$ is uniformly bounded. From this we conclude first that for any $1\leq n\leq N_T$, $\left(\eps^{-1}\vv^n_\eps\right)_{\eps>0}$ is uniformly bounded then that $\eps\vv^n_\eps\rightarrow0$ for any $0\leq n\leq N_T$. Therefore, taking the limit $\eps\rightarrow 0$, it yields that
for $0\leq n\leq N_T-1$,
$$
\eps^{-1}\,\vv^{n+1}_\eps \rightarrow \bE(t^n,\yy^n)\wedge \frac{\bB_{\rm ext}(t^n,
\yy^n)}{\|\bB_{\rm ext}\|^2}, \,\,{\rm when}\,\, \eps\rightarrow 0.
$$
Substituting the limit of $\eps^{-1}\,{\vv^{n+1}} $ in the first equation of
(\ref{scheme:0}) we prove that the limit $\yy^n$ satisfies (\ref{sch:y0}). Since the limit point $\yy_n$ is uniquely determined, actually all the sequence $(\xx^n_\eps)_{\eps>0}$ converges.
\end{proof}
\begin{remark}
The consistency provided by the latter result is far from being uniform with respect to the time step. However we do 
prove in the next section that the solution to (\ref{scheme:0}) is both uniformly
stable and  consistent with respect to $\Delta t$ and $\eps>0$. 
\end{remark}

 Of course, such a first order scheme is not 
accurate enough 
 to describe correctly the long time behavior of the numerical solution,
 but it has the advantage of the simplicity and we will prove in the
 next section that it is uniformly stable with respect to the
 parameter $\eps$ and the sequence 
$(\xx^n_\eps)$ converges to a consistent approximation 
 of the guiding center model when $\eps\rightarrow 0$.

Now, let us see how to generalize such 
an 
approach to second and third order schemes. 

\subsection{Second order semi-implicit Runge-Kutta schemes}
Now, we consider second order schemes with two stages. 
\subsubsection{A second order A-stable scheme}
A first example
of scheme satisfying the second order conditions is given by  a combination of Heun method (explicit part) and an
  $A$-stable second order singly diagonal  implicit Runge-Kutta SDIRK
  method (implicit part) \cite{hairer,BFR:15}. The first stage
  corresponds to
\begin{equation}
\label{scheme:2-1}
\left\{
\begin{array}{l}
\ds{\xx^{(1)} \,=\, \xx^n  \,+\, \frac{\Delta t}{2\eps}\,\vv^{(1)},}
\\
\,
\\
\ds{\vv^{(1)} \,=\, \vv^n  \,+\, \frac{\Delta t}{2\eps} \left[ \frac{\vv^{(1)}}{\eps} \wedge \bB_{\rm ext}(t^n, \xx^n) \,+\, \bE(t^n,\xx^n)\right].}
\end{array}\right.
\end{equation}

Then the second stage is given by
\begin{equation}
\label{scheme:2-2}
\left\{
\begin{array}{l}
\ds{\xx^{(2)} \,=\, \xx^n  \,+\, \frac{\Delta t}{2\eps}\,\vv^{(2)},}
\\
\,
\\
\ds{\vv^{(2)} \,=\, \vv^n  \,+\, \frac{\Delta t}{2\eps} \left[
  \frac{\vv^{(2)}}{\eps} \wedge \bB_{\rm ext}(t^{n+1}, 2\xx^{(1)} -\xx^n) \,+\, \bE(t^{n+1},2\xx^{(1)} -\xx^n)\right].}
\end{array}\right.
\end{equation}
Finally, the numerical solution at the new time step is
\begin{equation}
\label{scheme:2-3}
\left\{
\begin{array}{l}
\xx^{n+1} \,=\, \xx^{(1)}  + \xx^{(2)}  - \xx^n,
\\
\,
\\
\vv^{n+1} \,=\, \vv^{(1)}  + \vv^{(2)}  - \vv^n.
\end{array}\right.
\end{equation}

A similar numerical scheme has been proposed in the framework of
$\delta f$ simulation of the Vlasov-Poisson system  \cite{cheng2013}.
 
Under stability assumptions on the numerical solution to
(\ref{scheme:2-1})-(\ref{scheme:2-3}), we get the following
consistency result in the limit $\eps\rightarrow 0$.

\begin{proposition}[Consistency in the
  limit $\eps\rightarrow 0$ for a fixed $\Delta t$]
\label{prop:2}
Let us consider a time step $\Delta t>0$, a final time $T>0$  and
set $N_T=[T/\Delta t]$. Assume that the sequence $(\xx^n_\eps,\vv^n_\eps)_{0\leq
  n\leq N_T}$ given by (\ref{scheme:2-1})-(\ref{scheme:2-3}) is such that for all $1\leq n\leq N_T$,
$\left(\xx^n_\eps,\eps\vv^n_\eps\right)_{\eps>0}$ is uniformly
bounded with respect to $\eps>0$ and $\left(\xx^0_\eps,\eps\vv^0_\eps\right)_{\eps>0}$ converges in the limit $\eps\rightarrow 0$ to some $(\yy^0,0)$. 
Then, for $1\leq n\leq N_T$, $\xx^n_\eps\rightarrow \yy^n$, as $\eps\rightarrow 0$ and the limit $(\yy^n)_{n\geq 1}$ is a consistent and second
order approximation with respect to $\Delta t$ of the
guiding center equation given by the scheme
\begin{equation}
\label{sch:y1}
\left\{
\begin{array}{l}
\ds{\yy^{(1)} = \yy^n \,+\, \frac{\Delta t}{2} \left(\bE(t^n,\yy^n)\wedge \frac{\bB_{\rm ext}(t^n,
\yy^n)}{\|\bB_{\rm ext}\|^2}\right)},
\\
\,
\\
\ds{\yy^{(2)} = \yy^n \,+\, \frac{\Delta t}{2}\left(\bE(t^{n+1},2\yy^{(1)}-\yy^n)\wedge \frac{\bB_{\rm ext}(t^{n+1},
2\yy^{(1)}-\yy^n)}{\|\bB_{\rm ext}\|^2}\right)},
\end{array}\right.
\end{equation}
and $\yy^{n+1} = \yy^{(1)} + \yy^{(2)} - \yy^{n}$.
\end{proposition}
\begin{proof}
We follow the lines of the proof of Proposition~\ref{prop:1} and first choose a subsequence of $(\xx^n_\eps)$ converging to some $\yy^n$. The second equation of \eqref{scheme:2-1} implies that $\eps^{-1}\vv^{(1)}_\eps$ is bounded. From the first equation of \eqref{scheme:2-1} it follows that so is $\xx^{(1)}_\eps$. Then the second equation of \eqref{scheme:2-2} yields boundedness of $\eps^{-1}\vv^{(2)}_\eps$ and the first that $\xx^{(2)}_\eps$ is also bounded. Now from the second equation of \eqref{scheme:2-3} stems that $\eps\vv^n_\eps$ converges to zero as $\eps\to0$ for any $0\leq n\leq N_T$. Coming back to the second equation of \eqref{scheme:2-1} we show that for all $n$,
$\eps^{-1}\vv^{(1)}_\eps$ converges when $\eps\rightarrow 0$ and 
$$
\frac{\vv^{(1)}_\eps}{\eps}  \rightarrow \bE(t^n,\yy^n)\wedge \frac{\bB_{\rm ext}(t^n,
\yy^n)}{\|\bB_{\rm ext}\|^2}, \,\,{\rm when}\,\, \eps\rightarrow 0.
$$
Using the first equation of \eqref{scheme:2-1} we conclude that $\xx^{(1)}$ also converges and that its limit $\yy^{(1)}$ is given by the first equation of \eqref{sch:y1}.
Going on with the same arguments shows that
$$
\frac{\vv^{(2)}_\eps}{\eps}  \rightarrow \bE(t^{n+1},2\yy^{(1)}-\yy^n)\wedge \frac{\bB_{\rm ext}(t^{n+1},
2\yy^{(1)}-\yy^n)}{\|\bB_{\rm ext}\|^2}, \,\,{\rm when}\,\, \eps\rightarrow 0
$$ 
and that $\xx^{(2)}$ converges to a $\yy^{(2)}$ given by the second equation of \eqref{sch:y1}. One may then take a limit in the first equation of \eqref{scheme:2-3} and conclude that indeed the limit $\yy^n$ satisfies (\ref{sch:y1}). Again uniqueness supplies the convergence of the whole sequence.
\end{proof}


\subsubsection{A second order L-stable scheme}
Another choice is a combination of Runge-Kutta method (explicit part)
and  an $L$-stable second order SDIRK method in the implicit
part. This implicit scheme should give better stability properties on
the numerical solution with respect to the stiffness parameter $\eps>0$.

We first choose $\gamma>0$ as the
smallest root of the polynomial $\gamma^2 - 2\gamma + 1/2 = 0$, {\it
  i.e.} $\gamma = 1 - 1/\sqrt{2}$, then the scheme is given by the
following two stages. First, we have
\begin{equation}
\label{scheme:3-1}
\left\{
\begin{array}{l}
\ds{\xx^{(1)} \,=\, \xx^n  \,+\, \frac{\gamma\Delta t}{\eps}\,\vv^{(1)},}
\\
\,
\\
\ds{\vv^{(1)} \,=\, \vv^n  \,+\, \frac{\gamma\Delta t}{\eps}\,{\bf F}^{(1)},}
\end{array}\right.
\end{equation}
with
$$
{\bf F}^{(1)} \,:=\, \frac{\vv^{(1)}}{\eps} \wedge \bB_{\rm ext}(t^n, \xx^n) \,+\, \bE(t^n,\xx^n).
$$
For the second stage, we first define  
\begin{equation}
\label{tard}
\hat{t}^{(1)} \,:=\, t^n +
\frac{\Delta t}{2\gamma}, \quad\hat{\xx}^{(1)} \,:=\, \xx^{n}\,+\, \frac{\Delta t}{2\gamma\eps} \vv^{(1)},
\end{equation}
then  the solution $(\xx^{(2)},\vv^{(2)})$ is given by
\begin{equation}
\label{scheme:3-2}
\left\{
\begin{array}{l}
\ds{\xx^{(2)} \,=\, \xx^{n}  \,+\, \frac{(1-\gamma)\Delta t}{\eps} \,\vv^{(1)}  \,+\, \frac{\gamma\Delta t}{\eps}\,\vv^{(2)},}
\\
\,
\\
\ds{\vv^{(2)} \,=\, \vv^{n}  \,+\, \frac{(1-\gamma)\Delta t}{\eps}
  \,{\bf F}^{(1)}  \,+\, \frac{\gamma\Delta t}{\eps}  \,{\bf F}^{(2)}},
 \end{array}\right.
\end{equation}
with 
$$
{\bf F}^{(2)} \,:=\, \frac{\vv^{(2)}}{\eps} \wedge \bB_{\rm ext}\left(\hat{t}^{(1)}, \hat{\xx}^{(1)}\right) \,+\, \bE\left(\hat{t}^{(1)}, \hat{\xx}^{(1)}\right).
$$
Finally, the numerical solution at the new time step is
\begin{equation}
\label{scheme:3-3}
\left\{
\begin{array}{l}
\xx^{n+1} \,=\, \xx^{(2)},
\\
\,
\\
\vv^{n+1} \,=\, \vv^{(2)}.
\end{array}\right.
\end{equation}

Under stability assumptions on the numerical solution to
(\ref{scheme:2-1})-(\ref{scheme:2-3}), we get the following
consistency result in the limit $\eps\rightarrow 0$.

\begin{proposition}[Consistency in the
  limit $\eps\rightarrow 0$ for a fixed $\Delta t$]
\label{prop:3}
Let us consider a time step $\Delta t>0$, a final time $T>0$  and
set $N_T=[T/\Delta t]$. Assume that the sequence $(\xx^n_\eps,\vv^n_\eps)_{0\leq
  n\leq N_T}$ given by 
(\ref{scheme:3-1})-(\ref{scheme:3-3}) is such that for all $1\leq n\leq N_T$,
$\left(\xx^n_\eps,\eps\vv^n_\eps\right)_{\eps>0}$ is uniformly
bounded with respect to $\eps>0$ and $\left(\xx^0_\eps,\eps\vv^0_\eps\right)_{\eps>0}$ converges in the limit $\eps\rightarrow 0$ to some $(\yy^0,0)$. 
Then, for $1\leq n\leq N_T$, $\xx^n_\eps\rightarrow \yy^n$, as $\eps\rightarrow 0$ and the limit $(\yy^n)_{n\geq 1}$ is a consistent second
order approximation of the
guiding center equation, given by 
\begin{equation}
\label{sch:y2}
\left\{
\begin{array}{l}
\ds{{\bf U}^{n} = \bE(t^n,\yy^n)\wedge \frac{\bB_{\rm ext}(t^n,\yy^n)}{\|\bB_{\rm ext}\|^2}},
\\
\,
\\
\ds{\yy^{n+1} = \yy^{n} + {(1-\gamma)\Delta t} \,{\bf U}^{n} \,+\,
  \gamma\Delta t \,{\bf U}^{(1)}},
\end{array}\right.
\end{equation}
where
$$
\hat{\yy}^{(1)} := \yy^{n} + \frac{\Delta t}{2\gamma} \,{\bf
  U}^{n},\quad {\bf U}^{(1)} \,:=\, \bE(\hat{t}^{(1)},\hat{\yy}^{(1)})\wedge \frac{\bB_{\rm ext}(\hat{t}^{(1)},\hat{\yy}^{(1)})}{\|\bB_{\rm ext}\|^2}.
$$
\end{proposition}
%
We omit the proof of Proposition~\ref{prop:3} as almost identical to the one of Proposition~\ref{prop:2}.

The present scheme is $L$- stable, which means uniformly linearly stable with
respect to $\Delta t$.

\subsection{Third order semi-implicit Runge-Kutta schemes}
A  third order semi-implicit scheme is given by a four stages
Runge-Kutta method introduced in the framework of hyperbolic
systems with stiff source terms \cite{BFR:15}. First,  we set $\alpha=0.24169426078821$, $\beta =
\alpha/4$ and $\eta= 0.12915286960590$ and $\gamma=1/2-\alpha-\beta-\eta$. Then we construct the first
stage as 
\begin{equation}
\label{scheme:4-1}
\left\{
\begin{array}{l}
\ds{\xx^{(1)} \,=\, \xx^n  \,+\, \frac{\alpha\Delta t}{\eps}\,\vv^{(1)},}
\\
\,
\\
\ds{\vv^{(1)} \,=\, \vv^n  \,+\, \frac{\alpha\Delta t}{\eps} \,{\bf F}^{(1)},}
\end{array}\right.
\end{equation}
with 
$$
{\bf F}^{(1)} \,:=\, \frac{\vv^{(1)}}{\eps} \wedge \bB_{\rm ext}(t^n, \xx^n) \,+\,
\bE(t^n,\xx^n).
$$
For the second stage, we have
\begin{equation}
\label{scheme:4-2}
\left\{
\begin{array}{l}
\ds{\xx^{(2)} \,=\, \xx^{n}  \,-\, \frac{\alpha\Delta t}{\eps} \,\vv^{(1)}  \,+\, \frac{\alpha\Delta t}{\eps}\,\vv^{(2)},}
\\
\,
\\
\ds{\vv^{(2)} \,=\, \vv^{n}  \,-\, \frac{\alpha\Delta t}{\eps} \,{\bf
    F}^{(1)} \,+\, \frac{\alpha\Delta t}{\eps}\,{\bf F}^{(2)},}
\end{array}\right.
\end{equation}
with
$$
{\bf F}^{(2)} \,:=\, \frac{\vv^{(2)}}{\eps} \wedge \bB_{\rm ext}\left(t^{n}, \xx^{n}\right) \,+\, \bE\left(t^{n}, \xx^{n}\right).
$$
Then, for the third stage we set 
\begin{equation}
\label{scheme:4-3}
\left\{
\begin{array}{l}
\ds{\xx^{(3)} \,=\,   \xx^{n}  \,+\, \frac{(1-\alpha)\Delta t}{\eps}\,\vv^{(2)} \,+\, \frac{\alpha\Delta t}{\eps}\,\vv^{(3)},}
\\
\,
\\
\ds{\vv^{(3)} \,=\, \vv^{n} \,+\,  \frac{(1-\alpha)\Delta
    t}{\eps}\,{\bf F}^{(2)} \,+\, \frac{\alpha\Delta
  t}{\eps}\,{\bf F}^{(3)},}
\end{array}\right.
\end{equation}
with
$$
\left\{
\begin{array}{l}
\ds{{\bf F}^{(3)} \,:=\, \frac{\vv^{(3)}}{\eps} \wedge \bB_{\rm ext}\left(t^{n+1}, \overline{\xx}^{(2)}\right) \,+\, \bE\left(t^{n+1}, \overline{\xx}^{(2)}\right)},
\\
\,
\\
\ds{\overline{\xx}^{(2)} \,:=\, \xx^{n} + \frac{\Delta t}{\eps}\,\vv^{(2)}.}
\end{array}\right.
$$
Finally, for the fourth stage we set 
\begin{equation}
\label{scheme:4-4}
\left\{
\begin{array}{l}
\ds{\xx^{(4)} \,=\,  \xx^{n} \,+\, \frac{\beta\Delta t}{\eps}\,\vv^{(1)}\,+\, \frac{\eta\Delta t}{\eps}\,\vv^{(2)}\,+\, \frac{\gamma\Delta t}{\eps}\,\vv^{(3)}\,+\, \frac{\alpha\Delta t}{\eps}\,\vv^{(4)},}
\\
\,
\\
\ds{\vv^{(4)} \,=\, \vv^{n} \,+\, \frac{\beta\Delta
    t}{\eps}\,{\bf F}^{(1)}\,+\, \frac{\eta\Delta t}{\eps}\,{\bf F}^{(2)}\,+\,
  \frac{\gamma\Delta t}{\eps}\,{\bf F}^{(3)}\,+\, \frac{\alpha\Delta
    t}{\eps}\,{\bf F}^{(4)},}
\end{array}\right.
\end{equation}
with 
$$
\left\{
\begin{array}{l}
\ds{{\bf F}^{(4)} \,:=\,\frac{\vv^{(4)}}{\eps} \wedge
  \bB_{\rm ext}\left(t^{n+1/2}, \overline{\xx}^{(3)}\right) \,+\,
  \bE\left(t^{n+1/2}, \overline{\xx}^{(3)}\right),}
\\
\,
\\
\overline{\xx}^{(3)} \,:=\, \xx^{n}  \,+\, \frac{\Delta t}{4\eps}
\left( \vv^{(2)}  + \vv^{(3)} \right),
\end{array}\right.
$$
and the numerical solution at the new time step is
\begin{equation}
\label{scheme:4-5}
\left\{
\begin{array}{l}
\xx^{n+1} \,=\, \xx^{n}  \,+\, \frac{\Delta t}{6\eps} \left( \vv^{(2)} \,+\,
  \vv^{(3)}  \,+\, 4\, \vv^{(4)} \right),
\\
\,
\\
\vv^{n+1} \,=\, \vv^{n}  \,+\, \frac{\Delta t}{6\eps} \left( {\bf F}^{(2)} \,+\,
  {\bf F}^{(3)}  \,+\, 4 \,{\bf F}^{(4)} \right).
\end{array}\right.
\end{equation}
As for the previous schemes, under uniform stability assumptions with
respect to $\eps>0$, we prove the following Proposition
\begin{proposition}[Consistency in the
  limit $\eps\rightarrow 0$ for a fixed $\Delta t$]
\label{prop:5}
Let us consider a time step $\Delta t>0$, a  final time $T>0$  and
set $N_T=[T/\Delta t]$. Assume that the sequence $(\xx^n_\eps,\vv^n_\eps)_{0\leq
  n\leq N_T}$ given by 
(\ref{scheme:4-1})-(\ref{scheme:4-5}) is such that for all $1\leq n\leq N_T$,
$\left(\xx^n_\eps,\eps\vv^n_\eps\right)_{\eps>0}$ is uniformly
bounded with respect to $\eps>0$ and $\left(\xx^0_\eps,\eps\vv^0_\eps\right)_{\eps>0}$ converges in the limit $\eps\rightarrow 0$ to some $(\yy^0,0)$. 
Then, for $1\leq n\leq N_T$, 
$\xx^n_\eps\rightarrow \yy^n$, as $\eps\rightarrow 0$ and the limit $(\yy^n)_{n\geq 1}$ is a consistent  third
order approximation of the
guiding center equation provided by the scheme
\begin{equation}
\label{sch:y4}
\left\{
\begin{array}{l}
\ds{{\bf U}^{(2)} = \bE(t^n,\yy^{n})\wedge \frac{\bB_{\rm ext}({t}^{n},{\yy}^{n})}{\|\bB_{\rm ext}\|^2}},
\\
\,
\\
\ds{{\bf U}^{(3)} = \bE(t^{n+1},\yy^{(2)})\wedge \frac{\bB_{\rm ext}(t^{n+1},\yy^{(2)})}{\|\bB_{\rm ext}\|^2}},
\\
\,
\\
\ds{{\bf U}^{(4)} = \bE(t^{n+1/2},\yy^{(3)})\wedge \frac{\bB_{\rm ext}({t}^{n+1/2},{\yy}^{(3)})}{\|\bB_{\rm ext}\|^2}},
\end{array}\right.
\end{equation}
where
$$
\left\{
\begin{array}{l}
\yy^{(2)} \,=\, \yy^n \,+\, \Delta t\,{\bf U}^{(2)}, \\
\,
\\
\ds{\yy^{(3)} \,=\, \yy^n \,+\, \frac{\Delta t}{4}\,\left( {\bf U}^{(2)} \,+\, {\bf U}^{(3)} \right),}
\end{array}\right.
$$
and $$
\yy^{n+1} = \yy^{n} \,+\, \frac{\Delta t}{6} \,\left( {\bf U}^{(2)} \,+\, {\bf
    U}^{(3)} \,+\, 4\,{\bf U}^{(4)}\right).
$$
\end{proposition}
\begin{proof}
We follow the lines of the proof of Proposition \ref{prop:2}. In particular we prove successively the boundedness of $\eps^{-1}\vv^{(1)}_\eps$, $\xx^{(1)}_\eps$, $\eps^{-1}\vv^{(2)}_\eps$, $\xx^{(2)}_\eps$, $\eps^{-1}\vv^{(3)}_\eps$, $\xx^{(3)}_\eps$, $\eps^{-1}\vv^{(4)}_\eps$, $\xx^{(4)}_\eps$. Arguing inductively on $n$ one then shows that $\eps \vv^n_\eps$, ${\bf F}^{(1)}$, ${\bf F}^{(2)}$, ${\bf F}^{(3)}$, ${\bf F}^{(4)}$ all converge to zero when $\eps\to0$. Then for any converging subsequence of $\xx^n_\eps$ we may successively identify limits for $\eps^{-1}\vv^{(1)}_\eps$, $\xx^{(1)}_\eps$, $\eps^{-1}\vv^{(2)}_\eps$, $\xx^{(2)}_\eps$, $\eps^{-1}\vv^{(3)}_\eps$, $\xx^{(3)}_\eps$, $\eps^{-1}\vv^{(4)}_\eps$, $\xx^{(4)}_\eps$ and prove that any accumulating point of $\xx^n_\eps$ solves the limiting scheme involving \eqref{sch:y4}. Hence the result.
\end{proof}

\section{Analysis of the first-order semi-implicit scheme}
\setcounter{equation}{0}
\label{sec:4}

Consider the first order Euler semi-implicit scheme
\begin{equation}
\label{traj:ter}
\left\{
\begin{array}{l}
\ds\varepsilon\frac{\xx_k^{n+1} - \xx_k^{n}  }{\Delta t} = \vv_k^{n+1}, 
\\
\,
\\
\ds\varepsilon\frac{\vv_k^{n+1} - \vv_k^{n} }{\Delta t} = \frac{1}{\varepsilon}\vv_k^{n+1} \wedge {\bf B}_{\rm ext}(t^n,\xx_k^n)  + {\bf E}(t^n,\xx_k^n), 
\\
\,
\\
\xx^0_k = \xx_k^0, \, \vv^0_k = \vv^0_k\,.
\end{array}\right.
\end{equation}

We focus here on the case where both $\bE$ and $\bB_{\rm ext}$ are given external fields. 
Since, in this case, there is no coupling between particles we may focus safely on one of them and drop the ${}_k$ suffix.
Then we prove 
\begin{theorem}[Uniform consistency with respect to $\eps$]
We set
$$
\lambda \ :=\ \frac{\Delta t}{\eps^2}\,,
$$
assume that the $(\bE,\bB_{\rm ext}) \in
W_{loc}^{1,\infty}(\RR^+\times\RR^2)$ and consider the
solution $(\xx^n,\vv^n)$ to (\ref{traj:ter}). Then there exist 
positive constants $C$ and $\lambda_0$, only depending on $\bE$, $\bB_{\rm ext}$ and $t^n$ such that when\footnote{Obviously one may replace $1$ by any upper bound on $\Delta t$ but $\lambda_0$ and $C$ would depend on this upper bound.} 
$0\leq \Delta t\leq1$
and $\lambda\geq\lambda_0$
$$
\|\xx^n-\yy^n\|\ \leq\ \frac{C\Delta t}{\lambda}\ \left[ 1+ \left\|\eps^{-1}\vv^0-R[\bE(t^0,\xx^0)]\right\| \right],
$$ 
where $(\yy^n)$ corresponds to the numerical solution of the guiding
center model
\begin{equation}
\left\{
\begin{array}{l}
\ds\frac{\yy^{n+1} - \yy^{n}  }{\Delta t}\ =\
\bE(t^n,\yy^n)\wedge \frac{\bB_{\rm ext}}{\| \bB_{\rm ext}\|^2} , 
\\
\,
\\
\yy^0 = \xx^0\,.
\end{array}\right.
\label{cg:01}
\end{equation}
\label{th:1}
\end{theorem}

\begin{proof}
To start with we analyze the case where $\bE$ is a given bounded
Lipschitz field and $\bB_{\rm ext}$ is constant and of norm $1$.  Our estimates shall be expressed in terms of 
\begin{equation*}
K_0\ =\ \|\bE\|_{L^\infty}\,\quad K_t\ =\ \left\|\frac{\d\bE}{\d t}\right\|_{L^\infty}\,\quad K_x\ =\ \|\nabla_{\xx} \bE\|_{L^\infty}\,.
\end{equation*}
For concision's sake we also introduce the operator to denote $R[\bW]=\ \bW\wedge {\bf B}_{\rm ext}$.

By introducing the key quantity  for $n\geq1$,
$$
\zz^n= \eps^{-1}\vv^n-R[\bE(t^{n-1},\xx^{n-1})],
$$ 
the first equation of \eqref{traj:ter} reads 
$$
\ds \xx^{n}\ =\ \xx^{n-1} +\Delta t\ R[\bE(t^{n-1},\xx^{n-1})]+\Delta t\ \zz^n\,,\quad n\geq1
$$
while $\zz^1 = [\Id-\lambda R]^{-1}(\eps^{-1}\vv^0-R[\bE(t^0,\xx^0)])$ and
$$
\zz^{n+1} = [\Id-\lambda R]^{-1}(\zz^n-R[\bE(t^n,\xx^n)-\bE(t^{n-1},\xx^{n-1})])\,,\quad n\geq1\,.
$$

We first observe that $\|R\|\leq1$ and $[\Id-\lambda R]^{-1}=(\Id+\lambda R)/(1+\lambda^2)$ since $R^2=-\Id$.

This leads to
$$
\|\zz^{n+1}\|\ \leq\ \frac{1+\lambda}{1+\lambda^2}\ \left[(1+K_x\Delta t)\|\zz^n\|\ +\ (K_t+K_x\,K_0)\Delta t\right]\,,\quad n\geq1\,.
$$
Assuming 
$$
\frac{1+\lambda}{1+\lambda^2}(1+K_x\Delta t)\ <\ 1
$$
and introducing 
$$
a\ =\ \frac{1+\lambda}{1+\lambda^2}\ \frac{K_t+K_x\,K_0}{1-\frac{1+\lambda}{1+\lambda^2}(1+K_x\Delta t)}\Delta t
$$
we infer
$$
\|\zz^n\|\ \leq\ a\ +\ \left(\frac{1+\lambda}{1+\lambda^2}(1+K_x\Delta t)\right)^{n-1}\,b\,,\quad n\geq1\,,
$$
where
$$
b\ =\ a\ +\ \frac{1+\lambda}{1+\lambda^2} \left\|\frac1\eps\vv^0-R[\bE(t^0,\xx^0)]\right\|\,.
$$

For comparison we define $\yy^n$ solution to (\ref{cg:01}). Then,
\begin{eqnarray*}
\|\xx^n-\yy^n\| &\leq& (1+K_x\Delta t)\|\xx^{n-1}-\yy^{n-1}\|\ 
\\
&+&\ \Delta t\ \left[a\ +\ \left(\frac{1+\lambda}{1+\lambda^2}(1+K_x\Delta t)\right)^{n-1}\,b\right],\quad n\geq1\,.
\end{eqnarray*}
Hence assuming moreover $K_x>0$ (or replacing $K_x$ with some arbitrary positive number if $\|\nabla_x \bE\|_{L^\infty}=0$) 
\begin{eqnarray*}
\|\xx^n-\yy^n\| &\leq& \frac{a}{K_x}\left[\,1+(1+K_x\Delta t)^n\,\right]
\\
&+&\frac{b\,\Delta t}{(1+K_x\Delta t)[1- \frac{1+\lambda}{1+\lambda^2}]}\,\left[1+\left(\frac{1+\lambda}{1+\lambda^2}\right)^n\right]\,(1+K_x\Delta t)^n\,.
\end{eqnarray*} 
As a conclusion, for any $0\leq\theta<1$, there exists $C_\theta$ such that if
$$
\frac{1+\lambda}{1+\lambda^2}(1+K_x\Delta t)\ \leq\ \theta
$$
then, it yields
\begin{equation}
\|\xx^n-\yy^n\|\ \leq\ C_\theta\ \frac{\Delta t}{\lambda}\ \left[K_0+\frac{K_t}{K_x}+\left\|\frac1\eps\vv^0-R[\bE(t^0,\xx^0)]\right\|\right]\ e^{K_x\,n\Delta t}\,,
\end{equation}
which concludes the proof of Theorem~\ref{th:1} when the magnetic
field is uniform.

We now relax the 
assumption 
that $\bB_{\rm ext}$ is constant and of norm $1$. We set $\mathbf{b}_{\rm ext}=\|\bB_{\rm ext}\|$ and let $R$ be dependent on $t$ and $\xx$. We observe that now $R^2=-\mathbf{b}_{\rm ext}^2 \Id$ so that $[\Id-\lambda R]^{-1}=(\Id+\lambda R)/(1+\lambda^2\mathbf{b}_{\rm ext}^2)$. Then introducing the drift force
$$
\bF(t,\xx)\ =\ \frac{1}{\|\bB_{\rm ext}(t,\xx)\|^2}\ \bE(t,\xx)\wedge \bB_{\rm ext}(t,\xx)
$$
we essentially obtain the same estimates with $\frac{1+\lambda}{1+\lambda^2}$ replaced by
$$
\frac{1}{\Lambda}\ :=\ \left\|\frac{1+\lambda \mathbf{b}_{\rm ext}}{1+\lambda^2 \mathbf{b}_{\rm ext}^2}\right\|_{L^\infty} 
$$
and 
$$
K_0\ =\ \|\bF\|_{L^\infty}\,\quad K_t\ =\ \left\|\frac{\d\bF}{\d t}\right\|_{L^\infty}\,\quad K_x\ =\ \|\nabla_{\xx} \bF\|_{L^\infty}\,.
$$
Indeed introducing for $n\geq1$, 
$$
\zz^{n}=\frac{\vv^n}{\eps}-\bF(t^{n-1},\xx^{n-1}), 
$$
the scheme is written 
$$
\zz^1=[\Id-\lambda
R(t^0,\xx^0)]^{-1}\left(\frac{\vv^0}{\eps}-\bF(t^{0},\xx^{0})\right),
$$
and then
$$
\zz^{n+1}=[\Id-\lambda R(t^n,\xx^n)]^{-1}(\zz^n-(\bF(t^{n},\xx^{n})-\bF(t^{n-1},\xx^{n-1})))\,,\quad n\geq1
$$
together with
$$
\xx^{n}\ =\ \xx^{n-1} +\Delta t\ \bF(t^{n-1},\xx^{n-1})+\Delta t\ \zz^n\,,\quad n\geq1\,.
$$
Mark that if $\lambda\times(\inf\|\bB_{\rm ext}\|) > \sqrt{2}-1 $ then
$$
\Lambda\ \geq\ \frac{1+\lambda^2\times(\inf\|\bB_{\rm ext}\|)^2}{1+\lambda\times(\inf\|\bB_{\rm ext}\|)}\,.
$$
\end{proof}

Note that this result can be slightly improved when we modify the
initial condition of the asymptotic discrete  model. Indeed, consider 
$\yy^n$ solving
\begin{equation}
\left\{
\begin{array}{l}
\ds\frac{\yy^{n+1} - \yy^{n}  }{\Delta t}\ =\ \bE(t^n,\yy^n)\wedge \bB_{\rm ext} , 
\\
\,
\\
\yy^0 = \xx^0+\eps\,(\vv^0\wedge \bB_{\rm ext}+\eps\,\bE(t^0,\xx^0))\,.
\end{array}\right.
\end{equation}
The gain is that now
$$
\|\xx^1-\yy^1\|\ \leq\ \frac{\Delta t}{\lambda}\ \left[\,K_x\Delta t\,+\,\frac{1+\lambda}{1+\lambda^2}\,\right]\ \left\|\frac1\eps\vv^0-R[\bE(t^0,\xx^0)]\right\| 
$$
since 
$$
\zz^1\ =\ \frac1\lambda R\left[\frac1\eps\vv^0-R[\bE(t^0,\xx^0)]\right]\ -\  \frac1\lambda R\,[\Id-\lambda R]^{-1}\left[\frac1\eps\vv^0-R[\bE(t^0,\xx^0)]\right]\,.
$$
This leads, for $n\geq1$ to
\begin{eqnarray*}
\ds\|\xx^n-\yy^n\| &\leq&\ds \frac{a}{K_x}\left[\,1+(1+K_x\Delta t)^{n-1}\,\right]
\\
&+&\frac{b\,\Delta t}{1-
  \frac{1+\lambda}{1+\lambda^2}}\,\frac{1+\lambda}{1+\lambda^2}\,\left[1+\left(\frac{1+\lambda}{1+\lambda^2}\right)^{n-1}\right]\,(1+K_x\Delta
t)^{n-1}
\\
&+&\ds\frac{\Delta t}{\lambda}\ \left[\,K_x\Delta t\,+\,\frac{1+\lambda}{1+\lambda^2}\,\right]\ \left\|\frac1\eps\vv^0-R[\bE(t^0,\xx^0)]\right\|\ (1+K_x\Delta t)^{n-1}
\end{eqnarray*}
then with notation as above there exists a constant $C=C_\theta \,(K_0+\frac{K_t}{K_x}+K_x+1)e^{K_x
  t^n}>0$ such that
\begin{equation}
\|\xx^n-\yy^n\|\ \leq\ C\ \frac{\Delta t}{\lambda}\ \left[1+\left(\frac1\lambda+\Delta t\right)\left\|\frac1\eps\vv^0-R[\bE(t^0,\xx^0)]\right\|\right]\,.
\end{equation}

Concerning the analysis of high-order schemes, it is not
straightforward to adapt directly the strategy of Theorem
\ref{th:1}. Indeed,  the use of a semi-implicit scheme does not
necessarily guarantee  that the particle trajectories are under
control. 
\begin{remark}
Consider the scheme
(\ref{scheme:2-1})-(\ref{scheme:2-3}) in the simplest situation where
the electric field is zero and the magnetic field is ${\bf B}_{\rm
  ext}=(0,0,1)$, we show that the scheme preserves the kinetic energy
and we have
$$
\left\| {\vv^{n}} \right\|^2  = \left\| {\vv^{0}}
\right\|^2,  \quad \forall\, n\, \in\,\NN, 
$$ 
hence as $\eps$ goes to zero, the velocity $\vv^n/\eps$ cannot
converge to the null guiding center velocity except if the initial
velocity does. Fortunately, for the other
high order schemes (\ref{scheme:3-1})-(\ref{scheme:3-3}) and
(\ref{scheme:4-1})-(\ref{scheme:4-5}),  kinetic energy is  dissipated
and converges to 0.
\label{rem:nc}
\end{remark}

\section{Numerical simulations}
\label{sec:5}
\setcounter{equation}{0}

In this section, we discuss some examples to validate and to compare
the different time discretization schemes. We
first consider the single motion of a particle under the effect of a given
electromagnetic field. It allows us to  illustrate the
ability of the semi-implicit schemes
to capture the guiding center velocity with large time step $\Delta t $ in the
limit $\varepsilon\rightarrow 0$.

 Then we consider the Vlasov-Poisson system
with an external magnetic field.  A classical particle-in-cell method
is applied with different time discretization techniques to compute  the particle
trajectories. Hence this collection of charged particles move and
give rise to a self-consistent electric  field, obtained by solving
numerically the Poisson equation in a space grid. 

\subsection{One single particle motion}
Before going to the
statistical descriptions, let us investigate the accuracy and
stability properties of the semi-implicit algorithms presented in
Section~\ref{sec:3}  on the motion of individual particles in a given
electromagnetic field. 

Here we consider  an electric field $\bE=-\nabla\phi$, where
$$
\phi(\xx) = \frac{1}{2}\,\left(\|\xx\|^2 + \alpha\,\cos^2(2\pi\,y)\right), \quad \xx=(x,y)\in\RR^2,
$$
with $\alpha=0.02$ and a magnetic field
$\bB(\xx)=1\,+\,10^{-1}\,\sin(2\pi x)$ with $\xx=(x,y)\in\RR^2$. We
choose for all simulations $\Delta t=0.1$ and the initial data as $\xx^0=(1,1.4)$ and $\eps^{-1}\,\vv^0=(3,5)$, such that
the initial data $\zz^0=\eps^{-1}\vv^0 + {\bf F}(0,\xx^0)$ is bounded with respect to $\eps>0$.

\begin{figure}[!t]
\begin{center}
 \begin{tabular}{cc}
 \includegraphics[width=7.cm,height=7.25cm]{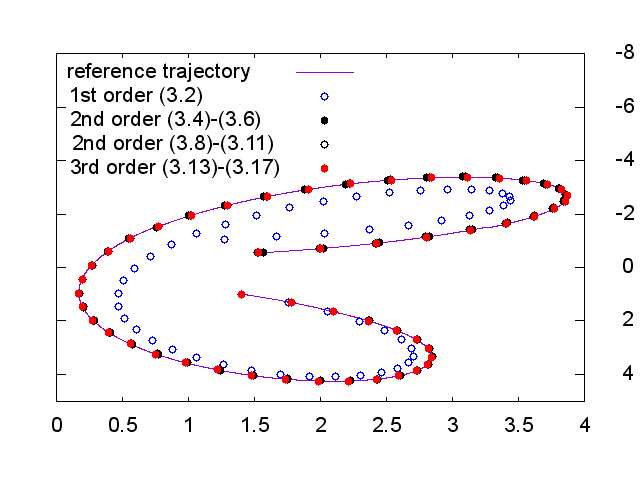} &    
\includegraphics[width=7.cm,height=7.25cm]{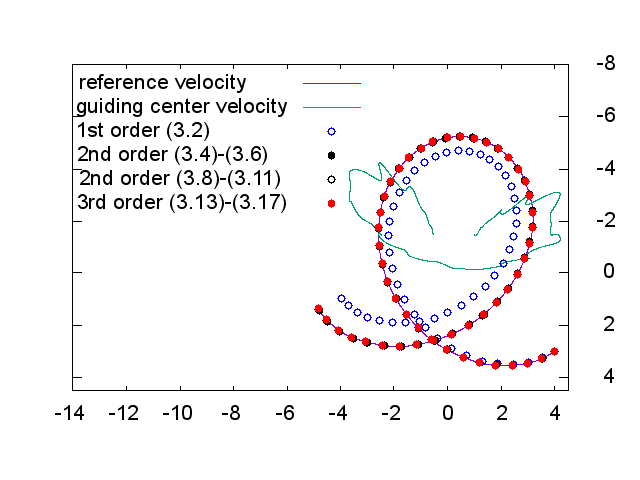} 
\\
\includegraphics[width=7.cm,height=7.25cm]{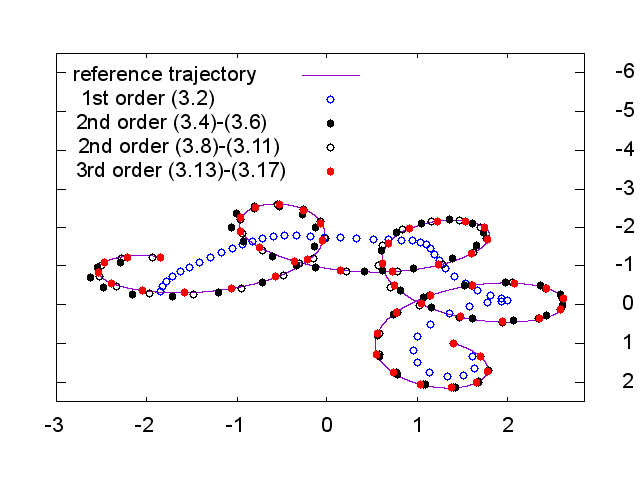} &    
\includegraphics[width=7.cm,height=7.25cm]{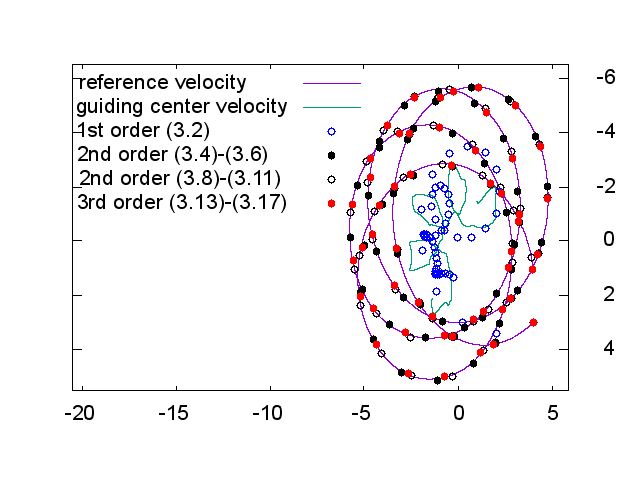} 
\\
(a)  & (b)  
\end{tabular}
\caption{\label{fig:single1}
{\bf One single particle motion.} Numerical solution obtained with a large time
step $\Delta t= 0.1$ with the 1st order scheme (\ref{scheme:0}), the
2nd order schemes 
(\ref{scheme:2-1})-(\ref{scheme:2-3})  and 
(\ref{scheme:3-1})-(\ref{scheme:3-3}) and the third order scheme
(\ref{scheme:4-1})-(\ref{scheme:4-5}) for  $\eps=1$ (top) and
$\eps=5.\,10^{-1}$ (bottom) : \\
(a) particle trajectory in physical space $(\xx^n)_{0\leq n\leq N_T}$\\
(b) particle velocity $(\vv^n)_{0\leq n\leq N_T}$.}
 \end{center}
\end{figure}

\begin{figure}[!t]
\begin{center}
 \begin{tabular}{cc}
\includegraphics[width=7.cm,height=7.25cm]{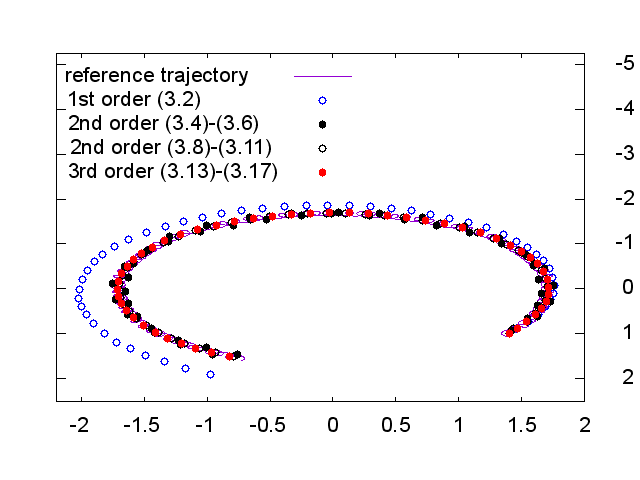} &    
\includegraphics[width=7.cm,height=7.25cm]{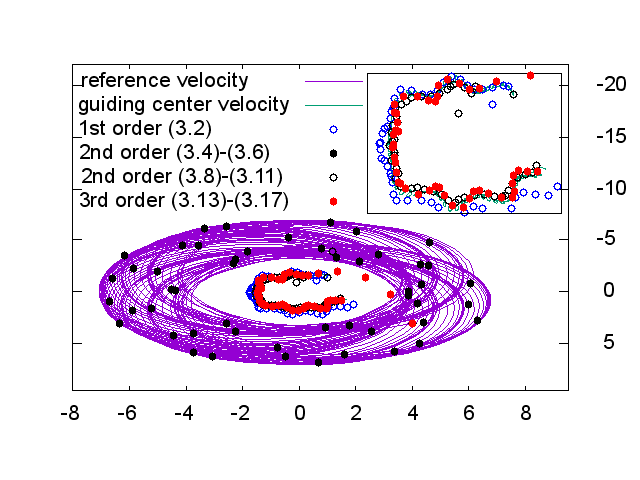} 
\\
\includegraphics[width=7.cm,height=7.25cm]{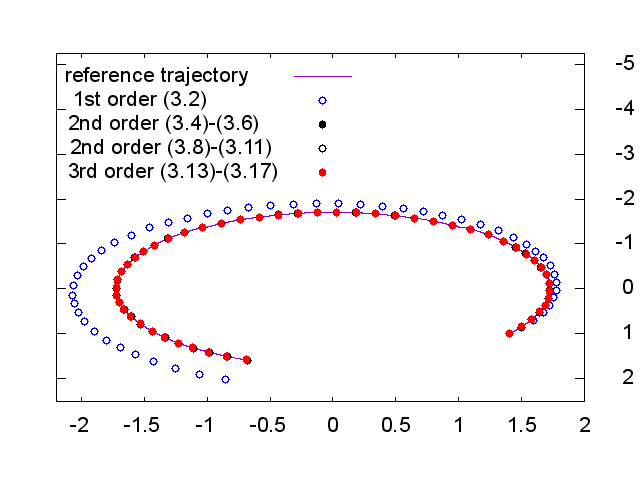} &    
\includegraphics[width=7.cm,height=7.25cm]{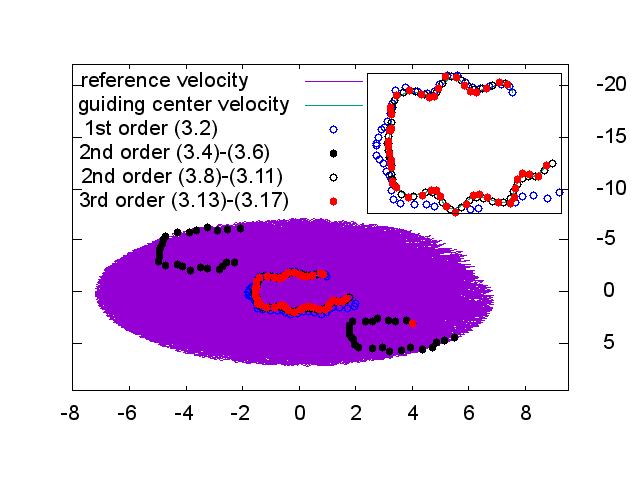} 
\\
(a)  & (b)  
\end{tabular}
\caption{\label{fig:single2}
{\bf One single particle motion.} Numerical solution obtained with a large time
step $\Delta t= 0.1$ with the 1st order scheme (\ref{scheme:0}), the
2nd order schemes
(\ref{scheme:2-1})-(\ref{scheme:2-3})  and 
(\ref{scheme:3-1})-(\ref{scheme:3-3}) and the third order scheme
(\ref{scheme:4-1})-(\ref{scheme:4-5}) for  $\eps=10^{-1}$ (top) and
$\eps=10^{-2}$ (bottom) : \\
(a) particle trajectory in physical space $(\xx^n)_{0\leq n\leq N_T}$\\
(b) particle velocity $(\vv^n)_{0\leq n\leq N_T}$.}
 \end{center}
\end{figure}

Thus,  we apply the schemes
proposed in Section\,\ref{sec:3} to compute a numerical solution to
(\ref{traj:bis})  and present the particle trajectory and velocity in Figures
\ref{fig:single1} and \ref{fig:single2}. These results are compared
with those  obtained with a
fourth order Runge-Kutta scheme using a small time step. 

In Figure \ref{fig:single1}, we first investigate the case where
$\eps$ is of order one ($\eps=1$ and $0.5$), which is the non stiff regime. On the left column, we
clearly observe that the space trajectory obtained from high order
schemes agree very well with the reference trajectory, whereas after
few time steps the first order scheme does not give satisfying
results. On the right hand side,  we present the evolution of the
velocity at each time step and compare it with the reference velocity
and the guiding center velocity. In this non stiff regime, the
velocity obtained from high order schemes coincides with the reference
velocity and the guiding center velocity is meaningless. Therefore,
this first test illustrates the ability of high order schemes to
describe accurately the particle motion in phase space when $\eps\sim 1$.
   
In Figure \ref{fig:single2}, we now propose the numerical results when
$\eps\ll 1$, that is, $\eps=0.1$ and $\eps=0.01$, which corresponds to
the high field regime. In that case, the space trajectory in the orthogonal plane to the magnetic field 
can be decomposed into a relatively slow motion due to the guiding
center velocity 
$$
\bF(t,\xx)\ =\ \frac{1}{\|\bB_{\rm ext}(t,\xx)\|^2}\ \bE(t,\xx)\wedge \bB_{\rm ext}(t,\xx)
$$
along the field line and a fast circular motion with a frequency of
order $1/\varepsilon$. Our aim here is to capture the slow motion using a fixed time step
$\Delta t=0.1$
independently of the value $\varepsilon\ll 1$. 

On the one hand, the space trajectory (left column of Figure \ref{fig:single2}) of the numerical
solution remains stable for various $\eps>0$ even if we do not solve
the fast scales. Moreover, when $\eps\rightarrow 0$, the numerical
solution approaches the correct trajectory and fast fluctuations are
somehow filtered thanks to the implicit treatment of the velocity
${\bf v}^n$. As in the previous simulations, we observe a discrepancy
between the first order scheme and other high order
schemes for large time. 

Furthermore, we focus on the velocity variable
$(\vv^n)_{0\leq n\leq N_T}$ and compare its time evolution with the
reference velocity and the guiding center velocity  (right column of Figure \ref{fig:single2}). Since we use
a large time step $\Delta t=0.1$, we cannot expect to follow all the details of
the velocity variable, but only the slow motion corresponding to the
guiding center velocity.  Now we clearly observe a different behavior
of the numerical solutions. On the one hand, the scheme
(\ref{scheme:2-1})-(\ref{scheme:2-3}), which preserves the kinetic
energy $\frac{1}{2}\,|\vv^n|^2$ when ${\bf E}=0$ and ${\bf B}_{\rm
  ext}=(0,0,1)$ (see Remark \ref{rem:nc}), still oscillates with a large amplitude   and the
velocity does not coincide with the guiding center velocity. On the
other hand, the schemes (\ref{scheme:0}),  (\ref{scheme:3-1})-(\ref{scheme:3-3}) and
(\ref{scheme:4-1})-(\ref{scheme:4-5}) are more dissipative and after
few time steps the velocity follows the line of the guiding center
velocity (see the zoom on the right column of Figure \ref{fig:single2}). Hence the amplitude of oscillations in the physical space are diminishing and the
particle follows the trajectory corresponding to the guiding center
model (\ref{eq:gc}).

As a  conclusion,  these elementary numerical simulations  confirm
the ability of the semi-implicit discretization to capture the slow
motion corresponding to the guiding center model (\ref{eq:gc})
uniformly with respect to $\varepsilon\ll 1$ and the interest of high
order  time discretization for the long time behavior of the solution.
   
\subsection{Diocotron instability}

We now consider the diocotron instability \cite{davidson}  for an annular electron
layer usually described by the guiding center model
(\ref{eq:gc}). This instability is well studied numerically as in \cite{filbet-yang, bibP}. It may give rise to electron vortices,
which is the analog of the Kelvin-Helmholtz fluid dynamics and may
occur when charge neutrality is not locally maintained. 

Here we want to investigate the development of such instability when
we consider the Vlasov-Poisson system with an external magnetic field
(\ref{eq:vlasov2d}). It is expected that such instability holds true
for large magnetic fields, whereas dissipative effects dominate when
the magnetic field is not large enough to confine
particles. 

Therefore, we perform numerical simulation using a
particle-in-cell method  \cite{birdsall} for the Vlasov equation
(\ref{eq:vlasov2d}), where the particle trajectories are approximated
thanks to the third order semi-implicit scheme
(\ref{scheme:4-1})-(\ref{scheme:4-5}). On the other hand, we compute
an approximation of the guiding center model using a finite difference
method developed in \cite{filbet-yang}. This reference solution will
be used to compare our results obtained from the Vlasov-Poisson system
with a large magnetic fields.

The initial density is given by
\begin{equation*}
 \rho_0(\xx)=
 \left\{
 \begin{array}{ll}
  (1+\alpha\cos(\ell\theta))\exp{(-4(\|\xx\|-6.5)^2)},&\text{if  }
  r^-\leq \|\xx \| \leq r^+,\\[3mm]
  0,&\text{otherwise},
 \end{array}
 \right.
\end{equation*}
where $\alpha$ is a small parameter, $\theta=\text{atan}(y/x)$, with
$\xx=(x,y)\in\RR^2$. In the following tests, we take $\alpha=0.01$, $r^-=5$, $r^+=8$, $\ell=7$.

Furthermore, we will also consider the Vlasov-Poisson system
(\ref{eq:vlasov2d}) with an external magnetic fiel with the initial datum $f_0$
$$
f_0(\xx,\vv) = \frac{\rho_0(\xx)}{2\pi}\, \exp\left(-\frac{\|\vv\|^2}{2}\right), \quad (\xx,\vv)\in\RR^4.
$$
A particle method with a third order semi-implicit solver
(\ref{scheme:4-1}))-(\ref{scheme:4-5}) will be applied for different
values of $\eps=  1$, $10^{-1}$ and $10^{-2}$.

For both systems a high order finite difference scheme in Cartesian coordinates
will be applied \cite{filbet-yang}  for the  numerical approximation
of
the Poisson equation in a disc domain. We choose a grid with $100\times 100$ points in the physical space and $400$ particles per cell for the discretization of the velocity space.

We define the total energy $\mathcal{E}(t)$ as the sum of the potential
energy $\mathcal{E}_{\rm pot}(t)$ and the kinetic energy $\mathcal{E}_{\rm
  kin}(t)$ with 
$$
\mathcal{E}_{\rm pot}(t) \,=\, \frac{1}{2}\int_{\RR^2} |\nabla\phi(t,\xx)|^2 d\xx,
\quad{\rm and}\quad \mathcal{E}_{\rm kin}(t) \,=\, \frac{1}{2}\int_{\RR^4} f(t,\xx,\vv)
\,|\vv|^2 d\xx d\vv.
$$ 
For the Vlasov-Poisson system (\ref{eq:vlasov2d}), the total energy is exactly conserved with respect to time.

First, we consider the case where $\eps=1$, that is, the particle
trajectories do not coincide with the trajectory corresponding to the
guiding center model (\ref{eq:gc}). On the one hand, we present in
Figures~\ref{fig:diocotron1}, the evolution of relative energy with respect to
the initial data  $t\mapsto \mathcal{E}(t)-\mathcal{E}(0)$ and the $L^\infty$ norm of the self-consistent
electric field $t\mapsto \|\bE(t)\|_{\infty}$. The discrete  total
energy is not exactly preserved but it oscillates around zero  and
these variations remain relatively small compared to the variations of
the kinetic and
potential energy which also oscillate with a frequency around
$1/2\pi$. The same phenomenon can be observed on the time evolution of
the $L^\infty$ norm of the electric field.
On the other hand  in Figure~\ref{fig:diocotron2}, we plot
the time evolution of the charged density and observe that when the
amplitude of the external magnetic field   $\|\bB_{\rm}\|$ is of order
one, the plasma is not well confined and does not develop any
instability. The density seems to  oscillate around the 
steady state.

\begin{figure}
\begin{center}
 \begin{tabular}{cc}
\includegraphics[width=7.cm, height=7.25cm]{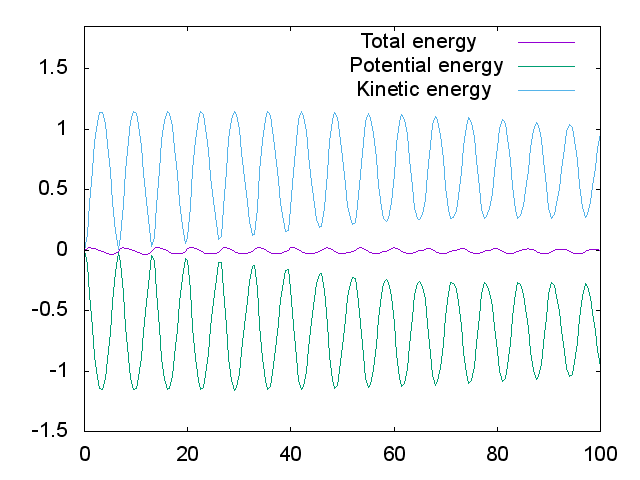}   &    
\includegraphics[width=7.cm,height=7.25cm]{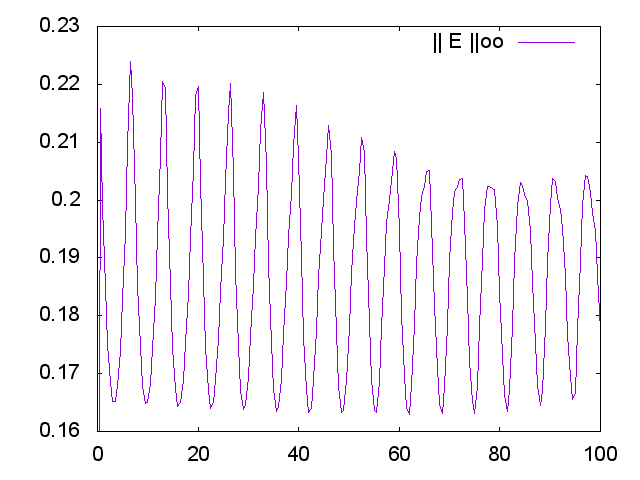}
\\
(a)  Relative Energy                               &    
(b) $\|\bE(t)\|_\infty$  
\end{tabular}
\caption{\label{fig:diocotron1}  {\bf Diocotron instability $\eps=1$.} Time evolution of (a) relative
  energy  $\mathcal{E}(t)-\mathcal{E}(0)$ with respect to the initial
  one and (b) $\|\bE(t)\|_\infty$. For $\eps \sim 1$ the instability
  does not occur.}
 \end{center}
\end{figure}

\begin{figure}
\begin{center}
 \begin{tabular}{cc}
\includegraphics[width=7.cm,height=7.25cm]{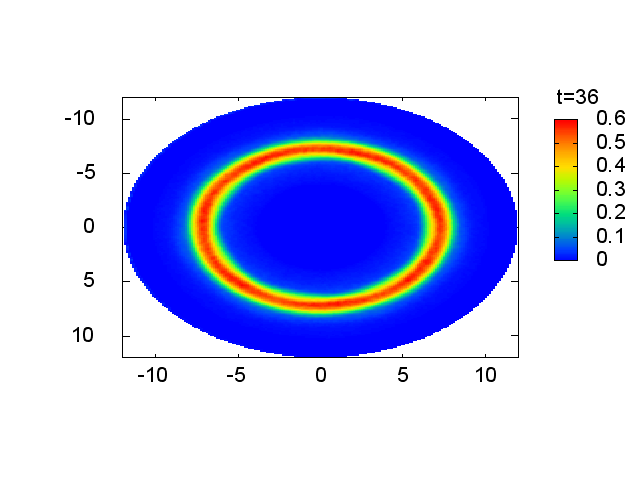}     &  
\includegraphics[width=7.cm,height=7.25cm]{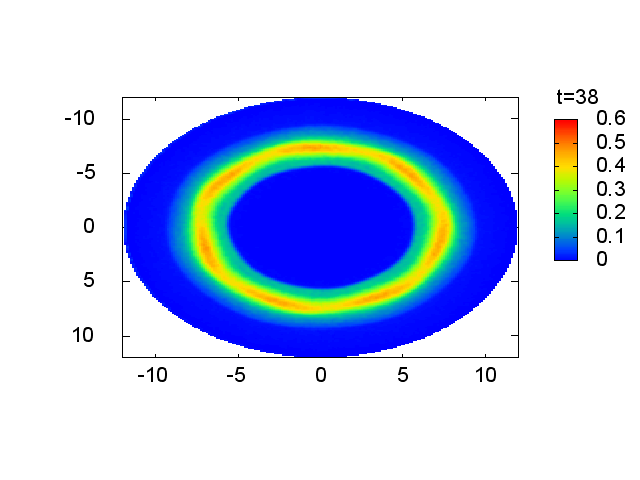}  
\\
\includegraphics[width=7.cm,height=7.25cm]{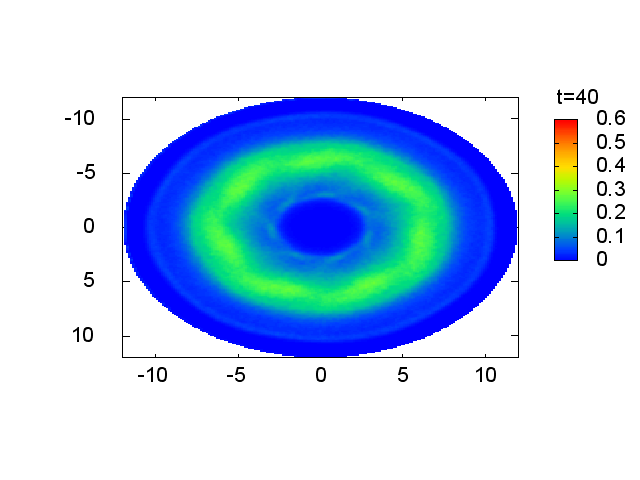}     &  
\includegraphics[width=7.cm,height=7.25cm]{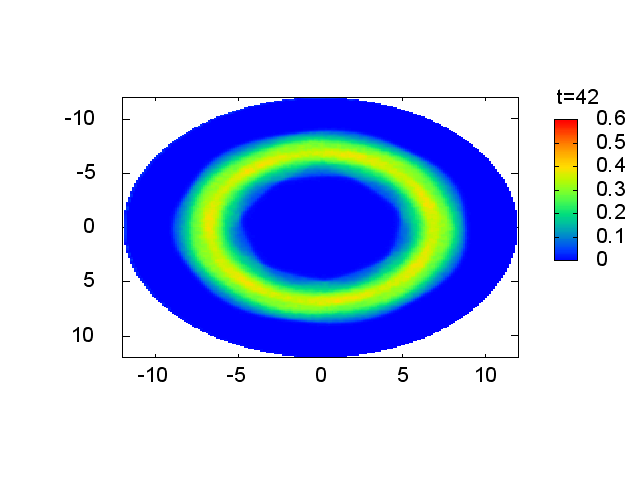}  
 \end{tabular}
\caption{\label{fig:diocotron2}  {\bf Diocotron instability $\eps=1$.} Time evolution of the density
  $\rho$ for time $t=36$, $t=38$, $t=40$ and $t=42$ units.  For $\eps \sim 1$ the instability
  does not occur.}
 \end{center}
\end{figure}

Then, we take a small value $\eps=0.01$, and for this case we present a comparison  between the finite difference approximation to the
guiding center model  (\ref{eq:gc}) and a particle method for the Vlasov-Poisson system
(\ref{eq:vlasov2d}), using the third order
semi-implicit scheme (\ref{scheme:4-1})-(\ref{scheme:4-5}). The
numerical results are presented in Figures~\ref{fig:diocotron3} and \ref{fig:diocotron4}.  

\begin{figure}
\begin{center}
 \begin{tabular}{cc}
\includegraphics[width=7.cm, height=7.25cm]{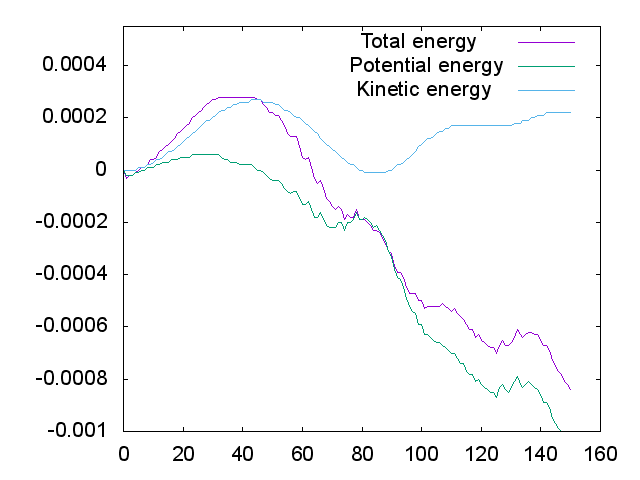}   &    
\includegraphics[width=7.cm,height=7.25cm]{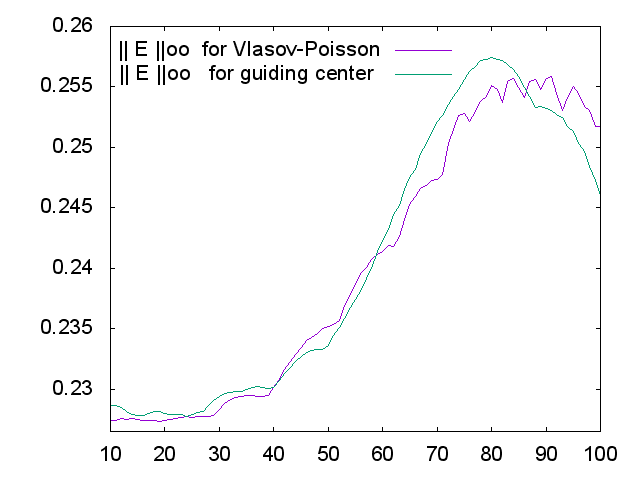}
\\
(a)  Relative Energy                               &    
(b) $\|\bE(t)\|_\infty$  
\end{tabular}
\caption{\label{fig:diocotron3}{\bf Diocotron instability $\eps=10^{-2}$.} Time evolution of
  the relative energy with respect to the initial one (a) for the Vlasov-Poisson
  system (\ref{eq:vlasov2d}) (b) for the guiding center model (\ref{eq:gc}).}
 \end{center}
\end{figure}

Let us emphasize  that for the guiding center model (\ref{eq:gc}), that is in the limit
$\eps\rightarrow 0$, the potential energy $\mathcal{E}_{\rm pot}$ is
conserved with respect to time. Therefore, for  $\eps\ll 1$, it is
expected that both the variations of $\mathcal{E}_{\rm pot}$ and
$\mathcal{E}_{\rm kin}$ are small.  In Figure \ref{fig:diocotron3}, we
can indeed observe that the variation of both quantities  $\mathcal{E}_{\rm pot}$ and
$\mathcal{E}_{\rm kin}$ are varying with a small amplitude of order
$10^{-3}$. In that 
case, 
we do not see any oscillations since a large
time step is used. Moreover, the second picture in
Figure \ref{fig:diocotron3} represents the time evolution of
$\|\bE(\cdot)\|_{\infty}$ for the Vlasov-Poisson system  (\ref{eq:vlasov2d}) and the guiding
center model (\ref{eq:gc}). Both results agree well which 
illustrates 
the accuracy of the particle method in the limit $\eps\ll 1$.

Finally in Figure \ref{fig:diocotron4}, we present a comparison between
the density $\rho$ obtained from the Vlasov-Poisson system discretized
with the particle method and the one corresponding to the guiding center model discretized
with a finite difference method.  These figures show the
development of the diocotron instability on the density $\rho$ for
both models. Indeed in this regime ($\eps\ll 1$), it is expected that the
density $\rho^\eps$ computed from the Vlasov-Poisson system obeys to
the same evolution as the one of the guiding center model. Once again, the results agree very well
and it is remarkable that the particle-in cell method does not suffer
from too many fluctuations. The vortices are well described even for large time
$t \sim 120$.

\begin{figure}
\begin{center}
 \begin{tabular}{cc}
\includegraphics[width=7.cm,height=7.cm]{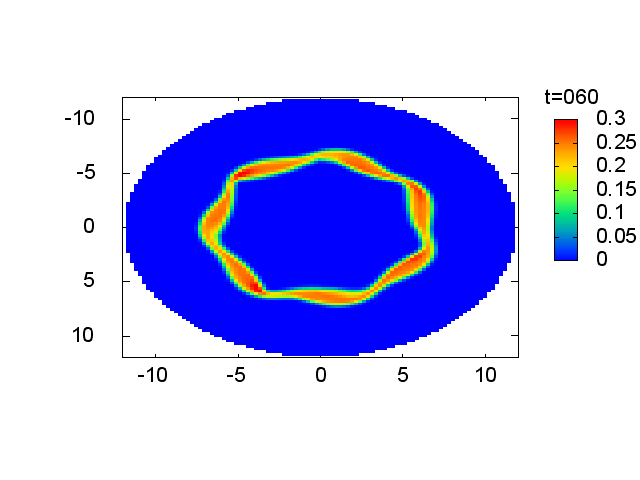}     &  
\includegraphics[width=7.cm,height=7.cm]{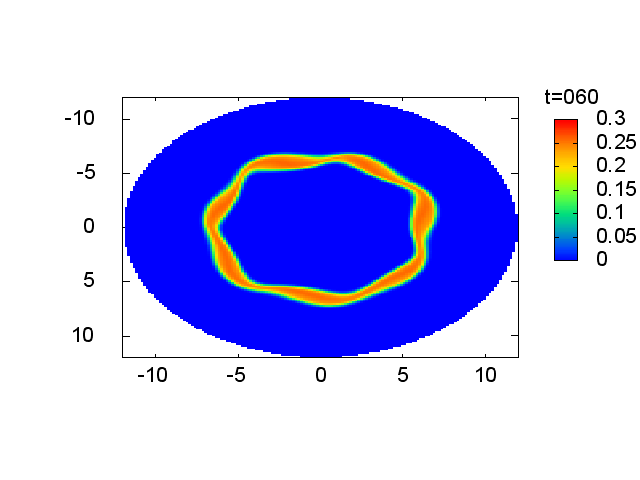}  
\\
\includegraphics[width=7.cm,height=7.cm]{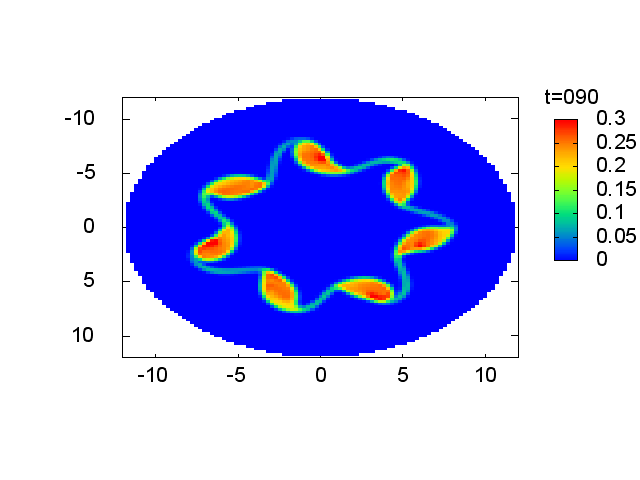}     &  
\includegraphics[width=7.cm,height=7.cm]{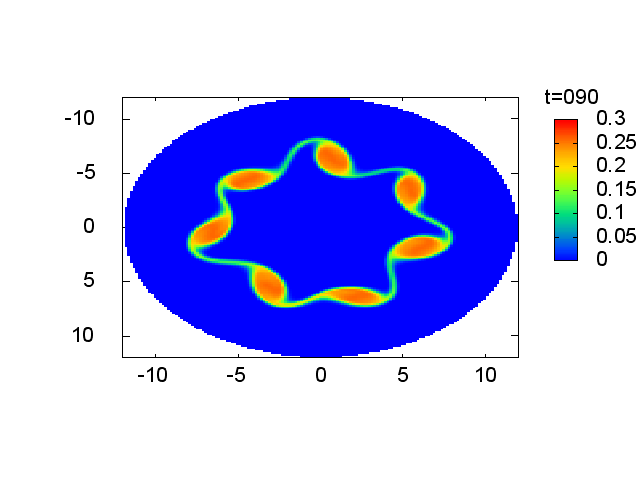}  
\\
\includegraphics[width=7.cm,height=7.25cm]{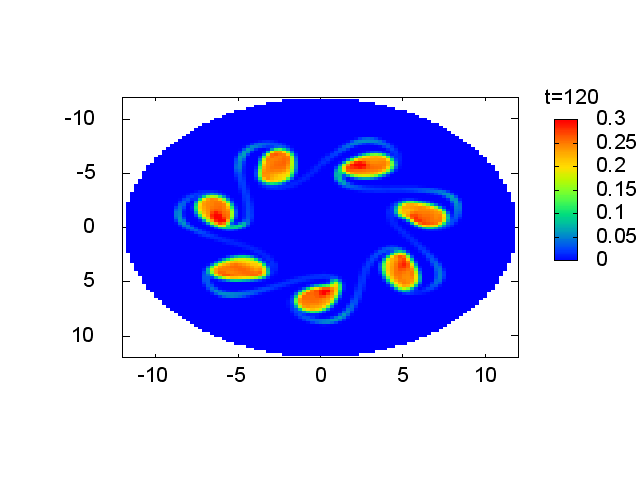}     &  
\includegraphics[width=7.cm,height=7.25cm]{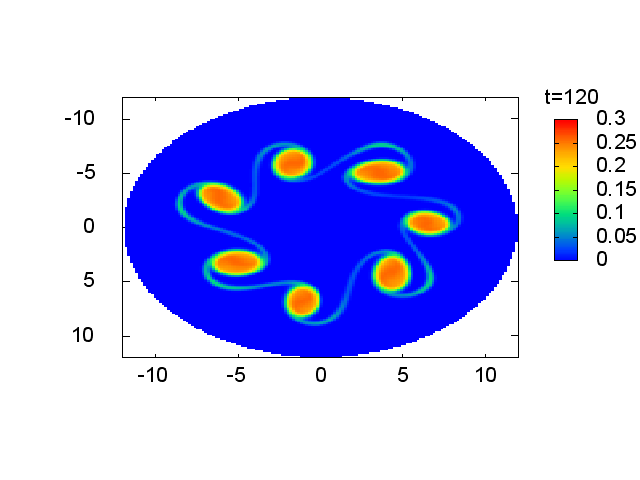}  
 \\
(a) & (b)
\end{tabular}
\caption{\label{fig:diocotron4} {\bf Diocotron instability $\eps=10^{-2}$.} Time evolution of the density
  $\rho$ for time $t=60$, $t=90$, and $t=120$ units for (a) the
  Vlasov-Poisson system (\ref{eq:vlasov2d}) and (b) the guiding center
  model (\ref{eq:gc}).}
 \end{center}
\end{figure}

\section{Conclusion and perspective}
\label{sec:6}
\setcounter{equation}{0}
In this paper we proposed a class of semi-implicit time discretization
techniques for particle-in cell simulations. The main feature of
this approach is to guarantee the accuracy and stability when the
amplitude of the magnetic field becomes large and to get the correct
long time behavior (guiding center approximation). We formally showed
that the present schemes preserve the initial order of accuracy when $\eps\rightarrow 0$. Furthermore,  we performed a complete analysis of the first order semi-implicit scheme
when we consider a given and smooth electromagnetic field
$(\bE,\bB_{\rm ext})$. 

The time discretization techniques proposed in this paper seem to be
a very simple and efficient tool to filter fast oscillations and have
nice stability and consistency properties in the limit
$\eps\rightarrow 0$. However, a  complete analysis of high
order semi-implicit schemes is still missing. The main issue is to control the space trajectory $(\xx^n)_n$ uniformly with
respect to $\eps$ and as we have shown in Remark \ref{rem:nc}, the use of a semi-implicit scheme does not
necessarily guarantee  that the particle trajectories are under
control. A complete analysis of high order schemes is currently under study.

On the other hand, the present techniques will be applied to
more advanced problems as the three dimensional Vlasov-Poisson system  when the magnetic
field is non uniform and the particle trajectories become more complicated.

\section*{Acknowledgements}
FF was supported by the EUROfusion Consortium and has received funding
from the Euratom research and training programme 2014-2018 under grant
agreement No 633053. The views and opinions expressed herein do not
necessarily reflect those of the European Commission.

LMR was supported in part by the ANR project BoND (ANR-13-BS01-0009-01).

\bibliographystyle{plain}

\begin{flushleft} \signFF \end{flushleft}
\begin{flushright} \signLMR \end{flushright}

\end{document}